\documentclass[12pt]{article}
\usepackage{mystyle}
\usepackage[a4paper,bindingoffset=0.2in,%
            left=1in,right=1in,
            top=1in,bottom=1in,%
            footskip=0.4in]{geometry}
\usepackage[style=numeric, backend=biber, doi=false, url=false, isbn=false, maxbibnames=99,
uniquename=false, giveninits]{biblatex}
\AtEveryBibitem{\clearlist{language}}
\renewbibmacro{in:}{%
  \ifentrytype{article}{}{\printtext{\bibstring{in}\intitlepunct}}}
\newcommand{\change}[1]{{ #1}}
\newcommand{\delete}[1]{{ }}
\newcommand*{\TitleFont}{%
      \usefont{\encodingdefault}{\rmdefault}{b}{n}%
      \fontsize{16}{20}%
      \selectfont}
\title{\TitleFont  Integrated density of states of the Anderson Hamiltonian with two-dimensional white noise}
\author{Toyomu Matsuda \footnote{Institut für Mathematik, Freie Universität Berlin, Arnimallee 7, 14195 Berlin, Germany; \\
\hspace*{1.7em}E-mail: \href{mailto:toyomu.matsuda@fu-berlin.de}{toyomu.matsuda@fu-berlin.de}}}
\date{}

\begin{document}
\maketitle
\begin{abstract}
  We construct the integrated density of states of the Anderson Hamiltonian with two-dimensional white noise
  by proving the convergence of the Dirichlet eigenvalue counting measures associated with the Anderson Hamiltonians on the boxes.
  We also determine the logarithmic asymptotics of the left tail of the integrated density of states.
  Furthermore, we apply our result to a moment explosion of the parabolic Anderson model in the plane.

  \vspace{\topsep}
  \noindent
  \href{https://mathscinet.ams.org/mathscinet/msc/msc2020.html}{\emph{MSC 2020}} -
                         60H25, %Random operators and Equations
                         60H20, %singular SPDEs
                         35J10, %Schrödinger operator, Schrödinger equation
                         35P20, %Asymptotic distributions of eigenvalues in context of PDEs
                         82B44. %disordered system
\end{abstract}
\tableofcontents
\section{Introduction}
{\bfseries Background.} The time evolution of a particle in quantum mechanics is described by a Schrödinger operator,
often given by $-\frac{1}{2} \Delta - V$ where $\Delta = \sum_{j=1}^d \partial_j^2$ is a Laplacian and $V$ is a potential.
We are interested in \emph{random Schrödinger operators}, also called \emph{Anderson Hamiltonians}, where the potentials are random.
Such operators appear in models where disorder need to be taken into consideration.
Mathematical study of random Schrödinger operators has a long history and is still active.
We refer the reader \change{for overview} to \change{Carmona and Lacroix} \cite{carmona1990spectral}, \change{Kirsch} \cite{kirsch_random_2008} and \change{Aizenman and Warzel} \cite{aizenman2015random}.

Anderson Hamiltonians can also be used to describe heat flow in disordered media.
The \emph{parabolic Anderson model} (PAM) is the \emph{stochastic partial differential equations} (SPDEs) of the following form:
\begin{equation*}
  \begin{cases}
    \partial_t u(t, x) = \frac{1}{2} \Delta u(t, x) + V(x) u(t, x) \quad \mbox{in } (0, \infty) \times X, \\
    u(0, \cdot) = u_0,
  \end{cases}
\end{equation*}
where $X$ is the space for the spatial variable $x$ and can be either discrete or continuous.
Then, the solution $u$ is described by the semigroup associated with the Anderson Hamiltonian $-\frac{1}{2} \Delta - V$.
The PAM is one of the most fundamental SPDEs and attracted attention from many mathematicians, \change{see the monograph König} \cite{konig2016parabolic}. a

The \emph{integrated density of states} (IDS) is  a fundamental object in the theory of random Schrödinger operators.
It heuristically measures the number of energy levels per unit volume below a given energy.
Given an Anderson Hamiltonian on $\R^d$, the IDS is defined as the limit
\begin{equation}\label{eq:convergence_to_IDS}
  N(\lambda) \defby \lim_{L \to \infty} \frac{1}{L^d} \sum_{n=1}^{\infty} \indic_{(-\infty, \lambda_{n,L}]}(\lambda),
\end{equation}
where $\{\lambda_{n, L}\}_{n=1}^{\infty}$ is the eigenvalues of the Anderson Hamiltonian on
$[-\frac{L}{2}, \frac{L}{2}]^d$.
It is a natural question whether the limit exists and is not trivial.
This question is positively answered for various models, see the references above, \change{Veselic} \cite{veselic2007existence} or \change{Kirsch and Metzger} \cite{ids2007}.
We note that properties of the IDS have been studied more than its construction.
For instance, the tail behavior of the IDS $N(\lambda)$ as $\lambda$ tends to the bottom of the spectrum, called \emph{Lifshitz tails},
has been extensively investigated for various models.

\delete{A natural choice of the random potential is an entirely uncorrelated one, namely white noise.
While such choice is common in discrete settings,
this is not the case in continuous settings.}
\change{In this paper, we are interested in white noise potential, that is, centered $\delta$-correlated Gaussian field, in continuous settings.
\emph{Formally}, this can be viewed as a limiting model of discrete Anderson Hamiltonian with i.i.d. potential.}
The problem is that white noise has low regularity and
can be realized only as a distribution-valued random variable.
Therefore, it is already not obvious to make sense of the Anderson Hamiltonian, let alone to construct the IDS.
Nevertheless, \change{Fukushima and Nakao} \cite{fukushima_spectra_1977}
constructed the IDS of the Anderson Hamiltonian with one-dimensional white noise by using the theory of Dirichlet forms.
Moreover, it rigorously proved the result from physicists
that the IDS has an exact form, which enables us to study the Lifshitz tails.
For later development of the Anderson Hamiltonian with one-dimensional white noise,
see \change{Dumaz and Labbé} \cite{dumaz_localization_2020} and the references given in its introduction.

Mathematical treatment of continuous Anderson Hamiltonians with white noise in higher dimensions
had been out of reach until the theory of regularity structures \cite{hairer_theory_2014} \change{of Hairer}
and the theory of paracontrolled distributions \cite{gubinelli_paracontrolled_2015} \change{of Gubinelli, Imkeller and Perkowski} emerged.
Both theories provide rigorous approaches to so-called \emph{singular SPDEs}, SPDEs inaccessible by classical approaches
due to the presence of ill-defined nonlinear terms or products.
For examples, they enable us to solve the PAM with white noise in two and three dimensions,
as shown in \cite{gubinelli_paracontrolled_2015}
and \change{Hairer and Labbé} \cite{hairer_Labbe_2015}, \cite{hairer_multiplicative_2018}.
There, the solution of the PAM is given by
the limit in probability of the solution of the PAM with the renormalized potential $V_{\epsilon} - c_{\epsilon}$
as $\epsilon \to 0+$, where
$V_{\epsilon}$ is smoothed white noise and $c_{\epsilon}$ is an appropriately chosen constant diverging to infinity.
It is believed that the PAM with white noise in four and higher dimensions cannot be solved.

Inspired by these theories, \change{Allez and Chouk} \cite{allez2015continuous} constructed continuous Anderson Hamiltonians with white noise on
two-dimensional toruses, followed by \change{Labbé} \cite{labbe_continuous_2019} and \change{Gubinelli, Ugurcan and Zachhuber} \cite{gubinelli_semilinear_2020}
to construct continuous Anderson Hamiltonians
with white noise in three dimensions and
by \change{Mouzard} \cite{mouzard2020weyl} to construct Anderson Hamiltonians on two-dimensional manifolds.
These works showed that the spectrum of the Anderson Hamiltonian consists of the eigenvalues.
Asymptotic behavior of the eigenvalues as the underlying boxes grow is studied in \change{Chouk and van Zuijlen} \cite{chouk2020asymptotics}, which is
crucially used in \change{König, Perkowski and van Zuijlen} \cite{konig2020longtime}
to prove the long time asymptotics of the mass of the two-dimensional PAM.
We also remark that
the work \cite{gubinelli_semilinear_2020} studies nonlinear wave equations and nonlinear Schrödinger equations whose
linear parts are given by the Anderson Hamiltonian with white noise in two and three dimensions.

\vspace{\topsep}
\noindent
{\bfseries Main result.}
The aim of this paper is to construct the IDS of the Anderson Hamiltonian with white noise in two dimensions and
determine the Lifshitz tail asymptotics.
Let $\{\lambda_{n,L}\}_{n=1}^{\infty}$ be the spectrum of the Anderson Hamiltonian on $Q_L \defby [-\frac{L}{2}, \frac{L}{2}]^2$
with Dirichlet boundary conditions. We denote by $N_L(\lambda)$
the distribution function $L^{-2} \sum_{n=1}^{\infty} \indic_{\{\lambda_{n,L} \leq \lambda\}}$.
The main theorem of this paper is the following:
\begin{theorem}\label{thm:main}
  Almost surely, the measures $\{N_L(d\lambda)\}_L$ converge vaguely to some limit $N(d\lambda)$.
  The measure $N(d\lambda)$ is deterministic and if we set $N(\lambda) \defby N((-\infty, \lambda])$,
  \begin{equation}\label{eq:lifshitz_tails_white_noise}
    \lim_{\lambda \to -\infty} \frac{1}{-\lambda} \log N(\lambda) = -\kappa^{-1},
  \end{equation}
  where $\kappa$ is the best constant of Ladyzhenskaya's inequality (a special case of the Gagliardo-Nirenberg inequality):
  \begin{equation}\label{eq:GN_ineq}
     \int_{\R^2} \abs{f(x)}^4 dx
    \leq C \int_{\R^2} \abs{f(x)}^2 dx \int_{\R^2} \abs{\nabla f(x)}^2 dx.
  \end{equation}
\end{theorem}
The proof is given in Theorem \ref{thm:convergence_to_IDS} and Theorem \ref{thm:lifshitz_tails_white_noise}.
The constant $\kappa$ has appeared in many related works. For instance, the work \cite{bass_self-intersection_2004} \change{of Bass and Chen} shows that
we have
\begin{equation*}
  \lim_{t \to \infty} \frac{1}{t} \log \P(\gamma([0,1]_{\leq}^2) \geq t) = -\kappa^{-1},
\end{equation*}
where $\gamma$ is the \emph{self-intersection local time} (SILT) of the two-dimensional Brownian motion and
$[0,1]_{\leq}^2 \defby \set{(s,t) \in [0,1]^2 \given s \leq t}$.
In \cite{chouk2020asymptotics}, the constant $\kappa$ is used to describe the eigenvalue asymptotics of the Anderson Hamiltonian
with two-dimensional white noise.
Both results will be crucially used in this work.

The moment of the PAM and the exponential moment of the SILT are related through the Feynman-Kac formula.
By exploiting this relation, \change{Gu and Wu} \cite{gu_moments_2018} proved that the PAM with two-dimensional white noise has
finite moment for small $t$ and conjectured that the moment explodes for large $t$.
As an easy application of Theorem \ref{thm:main} and ideas from its proof, the following is proved,
whose proof is given in Corollary \ref{cor:moment_of_pam}.
\begin{corollary}\label{cor:main}
  Let $u = u(t, x)$ be the solution of the PAM in $\R^2$ with white noise potential and with $u(0, \cdot) = \delta_0$.
  Then,
  \begin{equation*}
    \wnexpect[ u(t, 0) ]
    \begin{cases}
      < \infty \quad &\mbox{if } t \in (0, \kappa^{-1}), \\
      = \infty \quad &\mbox{if } t \in (\kappa^{-1}, \infty).
    \end{cases}
  \end{equation*}
\end{corollary}

We now discuss the strategy of proving Theorem \ref{thm:main}.
A standard approach to the proof of \eqref{eq:convergence_to_IDS} is the path integral approach (\cite{nakao_spectral_1977},
\cite[Section VI.1.2]{carmona1990spectral}), namely the approach to prove
the convergence of the Laplace transform.
An advantage of this approach is to obtain a representation of the Laplace transform.
The convergence of the Laplace transform can be proved as a consequence of the ergodic theorem applied to
the heat kernel of the Anderson Hamiltonian. There, the Feynman-Kac formula plays an important role as a technical tool.
However, in our setting, what the Feynman-Kac formula for the PAM with white noise should be is not obvious, since the white noise is not a function.
Although the work \cite{konig2020longtime} derived a Feynman-Kac representation for the PAM on a box, this is not useful for our problem.
To overcome the lack of useful Feynman-Kac formula, we view the heat kernel as the solution of the PAM with initial condition $\delta_0$ and
use moment estimate of the solution as technical input to apply the ergodic theorem.
For the moment estimate, we exploit the SILT in the spirit of \cite{gu_moments_2018}.
However, since we have to work on the PAM with initial condition $\delta_0$, we have to use the SILT of Brownian bridges.
It will be revealed that basic properties do not change between the SILT of the Brownian motion and that of Brownian bridges.

The proof of \eqref{eq:lifshitz_tails_white_noise} consists of two steps. We first show
\begin{equation*}
  \int_{\R} e^{-t\lambda} N(d\lambda)
  \begin{cases}
    < \infty, \quad &\mbox{if } t \in (0, \kappa^{-1}), \\
    = \infty, \quad &\mbox{if } t \in (\kappa^{-1}, \infty),
  \end{cases}
\end{equation*}
by applying the estimate of the exponential moment of the SILT of a Brownian bridge.
This readily implies $\limsup_{\lambda \to -\infty} (-\lambda)^{-1} \log N(\lambda) = -\kappa^{-1}$.
We next show
\begin{equation}\label{eq:liminf_lifshitz_tails_white_noise}
  \liminf_{\lambda \to -\infty} (-\lambda)^{-1} \log N(\lambda) \geq - \kappa^{-1}
\end{equation}
by using the tail estimate of
the first eigenvalue of the Anderson Hamiltonian obtained in \cite{chouk2020asymptotics}.

\change{
\begin{remark}\label{rem:gaussian_lifshitz}
  Let us describe the relation between our result and those of other Gaussian potentials. Fukushima and Nakao
  \cite{fukushima_spectra_1977} showed that the IDS $N^{\operatorname{1D}}$ with one-dimensional white noise satisfies
  $\log N^{\operatorname{1D}}(\lambda) \sim -c^{\operatorname{1D}} (-\lambda)^{\frac{3}{2}}$ as $\lambda \to -\infty$
  for some constant $c_{\operatorname{1D}} \in (0, \infty)$.
  As shown in Pastur~\cite{pastur_77}, Nakao~\cite{nakao_spectral_1977} and Kirsch and Martinelli~\cite{kirsch_ids_82},
  for a continuous centered Gaussain potential $U$, the IDS $N^U$ satisfies
  $\log N^U(\lambda) \sim - (2 \expect[U(0)^2])^{-1} (-\lambda)^2$.
  We will prove in Appendix~\ref{sec:lifshiz_tail_riesz} that for a centered Gaussian potential $V$ on $\R^d$ with
  \begin{equation*}
    \expect[V(x)V(y)] = \abs{x-y}^{-\alpha}, \quad \alpha < \min\{d, 2\},
  \end{equation*}
  the IDS $N^V$ satisfies $\log N^V(\lambda) \sim -c^V (-\lambda)^{\frac{4-\alpha}{2}}$.
  Thus, the decay of the Lifshitz tail is connected with the scaling property of the Gaussian potential.
  Our main result, Theorem~\ref{thm:main}, is distinct in that we establish this connection in the setting where renormalization is necessary.

  It turns out that leading coefficients of Lifshitz tails are also connected with covariances of Gaussian potentials. For this regard,
  see Remark~\ref{rem:variations}.

  For further topics such as restricted IDS and the Wegner estimate on Gaussian potentials, the reader is referred to the survey \cite{Leschke2005} by
  Leschke, Müller and Warzel.
\end{remark}
}

\vspace{\topsep}
\noindent
{\bfseries Structure of the paper. }Section \ref{sec:pam_and_ah} introduces notation and review basic results
on the PAM and the Anderson Hamiltonian with two-dimensional white noise.
In Section \ref{sec:brownian_bridge}, we study SILT of Brownian bridges.
The goal of this section is to extend some well-known results of the SILT of the
Brownian motion to those of Brownian bridges. The reader who is not interested in
technical aspects can skip this section.
Section \ref{sec:ids} is the main part of this
paper as it gives the proof of Theorem \ref{thm:main}.
In Appendix \ref{sec:lifshiz_tail_riesz}, we study the Lifshitz tails of the IDS with the Gaussian potential whose covariance is
given by a Riesz potential.
Appendix \ref{sec:gaussian} is for technical computations related to the white noise. Finally, Appendix \ref{sec:laplace_transform}
contains two elementary lemmas on Laplace transforms.

\vspace{\topsep}
\noindent
{\bfseries Notation. }
\begin{itemize}
  \item $\N_0 \defby \{0, 1, 2, \ldots\}$, $\N \defby \{1, 2, \ldots\}$.
  \item $[a,b]^2_{\leq} \defby \set{(s,t) \in [a, b]^2 \given s \leq t}$, $Q_L \defby [-\frac{L}{2}, \frac{L}{2}]^2$.
  \item $\delta_x$ is the Dirac's delta function: $\int_{\R^2} f(y) \delta_x(y) dy = f(x)$.
  \item $p_t(x) \defby \frac{1}{2\pi t} \exp(-\frac{\abs{x}^2}{2t})$, $t\in (0, \infty)$, $x \in \R^2$.
  \item $c_{\epsilon} \defby \frac{1}{2\pi} \log \epsilon^{-1}$.
  \item $B$ is the $d$-dimensional Brownian motion with $B_0 = 0$ and
        $X^t$ is the $d$-dimensional Brownian bridge with $X^t_0 = X^t_t = 0$.
        We set $d=2$ except for Appendix \ref{sec:lifshiz_tail_riesz}.
        We denote by $B_{[0,t]}$ the set $\set{B(s) \given s \in [0, t]}$ and use similar notation for $X^t$.
  \item We use the notation $\P$ and $\expect$ for the probability and the expectation with respect to $B$ or $X^t$.
        We use the notation $\wnP$ and $\wnexpect$ for the probability and the expectation with respect to the white noise.
\end{itemize}

\section{PAM and Anderson Hamiltonian}\label{sec:pam_and_ah}
In this section, we introduce notation and review fundamental results on the PAM and the Anderson Hamiltonian
with two-dimensional white noise.
Some technical results on the white noise are given in Appendix \ref{sec:gaussian}.
\begin{definition}\label{def:white_noise}
  We denote by $\S$  the space of Schwartz functions on $\R^2$ and denote its dual by $\S'$.
  \begin{enumerate}[(i)]
    \item Let $\P^{\operatorname{WN}}$ be the probability measure on $\S'$ under which $f \mapsto \inp{f}{\phi}$ is
    centered Gaussian with variance $\norm{\phi}_{L^2(\R^2)}$ for $\phi \in \S$.
    We set $\wnexpect[\cdot] \defby \expect_{\P^{\operatorname{WN}}}[\cdot]$ and
    denote by $\xi$ the identity map in $\S'$. The random variable $\xi$ is called the \emph{white noise} in $\R^2$.
    We set $\xi_{\epsilon} \defby \xi \conv p_{\frac{\epsilon}{2}}$,
    where $p$ is the heat kernel given in Notation.
    \item For $x \in \R^2$, we define the map $T_x: \S' \to \S'$ by $\inp{T_x f}{\phi} \defby \inp{f}{\phi(\cdot - x)}$.
  \end{enumerate}
\end{definition}
We introduce notation on function spaces from \cite{chouk2020asymptotics}.
\begin{definition}
  Let $L \in (0, \infty)$, $p, q \in [1, \infty]$ and $\alpha \in \R$.
  \begin{enumerate}[(i)]
    \item As defined in \cite[Definition 4.9]{chouk2020asymptotics}, we denote
    by $B_{p, q}^{\diri, \alpha}(Q_L)$ the Dirichlet Besov space
    and denote by
    $B_{p,q}^{\n, \alpha}(Q_L)$ be the Neumann Besov space.
    We write $\csp_L^{\#,\alpha} \defby B_{\infty, \infty}^{\#, \alpha}(Q_L)$ for
    $\# \in \{\diri, \n\}$.
    \item  Let $(\n_{k, L})_{k \in \N_0^2}$ be the Neumann basis of $L^2(Q_L)$ given by
      \begin{equation*}
        \n_{k,L}(x) \defby c_{k, L} \indic_{Q_L}(x) \cos\Big\{\frac{\pi k_1}{L} \Big(x_1 + \frac{L}{2} \Big)\Big\}
        \cos\Big\{\frac{\pi k_2}{L} \Big(x_2 + \frac{L}{2}\Big)\Big\}
        \quad x = (x_1, x_2) \in \R^2,
      \end{equation*}
    where $c_{k, L}$ is the normalizing constant and let $\sigma_L(D)$ be the Fourier multiplier defined by
    \begin{equation*}
      \sigma_L(D) f \defby \sum_{k \in \N^2_0} (1 + \pi^2 L^{-2} \abs{k}^2)^{-1} \inp{f}{\n_{k,L}} \n_{k,L}.
    \end{equation*}
    \item
    We write $\X_{L}^{\n, \alpha}$ for the space of enhanced Neumann potentials, defined as the closure of the set
    \begin{equation*}
      \set{(f, f \resonant \sigma_L(D) f - c) \given f \in \S_{\n}, c \in \R}
      \quad \mbox{in } \csp^{\n, \alpha}_{L}
    \times \csp^{\n, 2 \alpha+ 2}_{L},
    \end{equation*}
    where the symbol $\resonant$ is the resonant product \cite[Definition 4.24]{chouk2020asymptotics} and $\S_{\n}$ is the space of evenly extended smooth functions
    \cite[Definition 4.4]{chouk2020asymptotics}.
    \item For $\bm{\theta} \in \X_L^{\n, \alpha}$, we write the product $\bm{\theta} \diamond f$ as in
    Definition \cite[Definition 5.3]{chouk2020asymptotics}. The reader should be noted that the symbol $\diamond$ has nothing to do with the Wick product.
  \end{enumerate}
\end{definition}
The next definition is concerned with the enhanced white noise.
\begin{definition}\label{def:enhanced_white_noise}
  Recall the definition of the constant $c_{\epsilon}$ from Notation.
  \begin{enumerate}[(i)]
    \item Let $L \in (0, \infty)$ and $\epsilon \in (0, 1)$. We set
    \begin{equation*}
      T_x \xi_{L, \epsilon}^{\n} \defby \sum_{k \in \N_0^2} \inp{T_x \xi_{\epsilon}}{\n_{k,L}} \n_{k,L},
      \quad \xi^{\n}_{L, \epsilon} \defby T_0 \xi^{\n}_{L, \epsilon},
    \end{equation*}
    \begin{equation*}
      T_x \Xi_{L, \epsilon}^{\n} \defby T_x \xi^{\n}_{L, \epsilon} \resonant \sigma_L(D) T_x \xi^{\n}_{L, \epsilon} - c_{\epsilon},
      \quad \Xi_{L, \epsilon}^{\n} \defby T_0 \Xi_{L, \epsilon}^{\n},
    \end{equation*}
    $T_x \bm{\xi}_{L, \epsilon}^{\n} \defby (T_x \xi_{L, \epsilon}^{\n}, T_x \Xi_{L, \epsilon}^{\n})$
    and $\bm{\xi}_{L, \epsilon}^{\n} \defby T_0 \bm{\xi}_{L, \epsilon}^{\n}$.
    \item We fix a countable and unbounded set $\mathbb{I}$ of $[1, \infty)$ such that $rL \in \mathbb{I}$ for every
    $r \in \Q \cap [1, \infty)$ and $L \in \mathbb{I}$. By Proposition \ref{prop:unif_conv_of_enhanced_noise},
    there exist a set $\mathbb{J} \subseteq \Q \cap (0, 1)$ with $\inf \mathbb{J} = 0$ and a $\wnP$-full set $\S'_0 \subseteq \S'$
    such that for every $\omega \in \S_0'$, $L \in \mathbb{I}$, $M \in [1, \infty)$ and $\alpha \in (-\frac{4}{3}, -1)$,
    \begin{equation*}
      \lim_{\epsilon, \epsilon' \in \mathbb{J}, \epsilon, \epsilon' \to 0+} \sup_{x \in Q_M}
      \norm{T_x \bm{\xi}_{L, \epsilon}^{\n}(\omega) - T_x \bm{\xi}_{L, \epsilon'}^{\n}(\omega)}_{\X_{L}^{\n, \alpha}} = 0.
    \end{equation*}
    \item For $L \in \mathbb{I}$ and $x \in \R^2$, we set $T_x \bm{\xi}_L^{\n}(\omega) \defby \lim_{\epsilon \in \mathbb{J},
    \epsilon \to 0+} T_x \bm{\xi}_{L, \epsilon}^{\n}(\omega)$ on $\S'_0$
    and $\bm{\xi}_L^{\n} \defby T_0 \bm{\xi}_L^{\n}$.
    \end{enumerate}
\end{definition}

\begin{remark}
In the rest of this section and Section \ref{sec:ids},
we assume that $L$ and $\epsilon$ belong to $\mathbb{I}$ and $\mathbb{J}$ respectively and
that $\xi$ is the identity map in $\S'_0$, so that we will not be annoyed with exceptional $\wnP$-zero sets.
For instance, $\lim_{L \to \infty}$ means $\lim_{L \in \mathbb{I}, L \to \infty}$.
\end{remark}

Let $\alpha \in (-\frac{4}{3}, -1)$, $\gamma \in (-1, 2+\alpha)$, $p \in [1, \infty]$, $\bm{\theta} \in \X_{L}^{\n, \alpha}$
and $\phi \in B_{p, \infty}^{\diri, \gamma}(Q_L)$. The theory of paracontrolled distributions \cite{gubinelli_paracontrolled_2015} enables us
to solve the following PDE with Dirichlet boundary conditions:
\begin{equation*}
  \begin{cases}
    \partial_t u_L^{\phi, \bm{\theta}} =
    \frac{1}{2} \Delta u_L^{\phi, \bm{\theta}} + \bm{\theta} \diamond u_L^{\phi, \bm{\theta}}  \,\,
     \mbox{ in } (0, \infty) \times Q_L, \\
    u_L^{\phi, \bm{\theta}}(0, \cdot) = \phi, \quad \mbox{ and }\,\,\, u_L^{\phi, \bm{\theta}}(t, \cdot) = 0 \,\, \mbox{ on }
    (0, \infty) \times \partial Q_L.
  \end{cases}
\end{equation*}
See \cite[Section 2]{konig2020longtime} for details.
\begin{definition}
  For $L \in \mathbb{I}$, $\epsilon \in \mathbb{J}$, $x \in Q_L$ and $y \in \R^2$,
  we set
  \begin{equation*}
    u_{L, \epsilon}^{x,y} \defby u_L^{\delta_x, T_y \bm{\xi}_{L, \epsilon}^{\n}} \quad \mbox{and} \quad
    u_{L}^{x,y} \defby u_L^{\delta_x, T_y \bm{\xi}_{L}^{\n}}.
  \end{equation*}
\end{definition}
With the parameters as above, \cite[Lemma 2.8]{konig2020longtime} implies the map
\begin{equation*}
  B_{p, \infty}^{\diri, \gamma}(Q_L) \times \X_{L}^{\n, \alpha} \ni (\phi, \bm{\theta})
  \mapsto u_L^{\phi, \bm{\theta}}(t, \cdot) \in \csp^{\diri, \beta}_{L}
\end{equation*}
 is continuous for
$\beta \in (0, \alpha + 2)$.
In particular,
we have $\lim_{\epsilon \to 0+} u_{L, \epsilon}^{x,y}(t, \cdot) = u_L^{x, y}(t, \cdot)$ in $\csp^{\diri, \beta}_L$.

Recall that, as mentioned in Notation, the Brownian motion $B$ in $\R^2$ starts
from $0$.
By the Feynman-Kac formula \cite[Section II.3.2]{carmona1990spectral}, we have
\begin{equation}\label{eq:feynman-kac}
  u_{L, \epsilon}^{x, y} (t, z) = p_t(x-z) \expect \Big[ \exp\Big( \int_0^t \{T_y \xi_{\epsilon}(z + B_s) - c_{\epsilon}\} ds \Big)
  \indic_{\{z + B_{[0, t]} \subseteq Q_L\}} \Big\vert B_t = x - z \Big]
\end{equation}
for $t \in (0, \infty)$, $x, z \in Q_L$ and $y \in \R^2$.
Therefore, if $L_1 \leq L_2$ then $u_{L_1, \epsilon}^{x,y}(t, z) \leq u_{L_2, \epsilon}^{x, y}(t, z)$ and hence
$u_{L_1}^{x,y}(t, z) \leq u_{L_2}^{x, y}(t, z)$. This implies the limits
in the following definition make sense.
\begin{definition}\label{def:def_of_u}
  For $t \in [0, \infty)$ and $x, y, z \in \R^2$, we set
  \begin{equation*}
    u_{\infty}^{x,y}(t, z) \defby \lim_{L \to \infty} u_{L}^{x,y}(t, z) \in [0, \infty],
    \quad u(t, z) \defby u_{\infty}^{0, 0}(t, z).
  \end{equation*}
\end{definition}
\begin{lemma}\label{lem:shift_of_u_infty}
  For every $t \in [0, \infty)$ and $x, y, z \in \R^2$, we have $u_{\infty}^{x, y}(t, x) = u_{\infty}^{x-z, y+z}(t, x-z)$.
\end{lemma}
\begin{proof}
  For $x, z \in Q_{L_0}$ and $L \geq L_0$, the Feynman-Kac formula \eqref{eq:feynman-kac} implies
  \begin{align*}
    u_{L, \epsilon}^{x, y}(t, x)
    &= \frac{1}{2\pi t} \expect \Big[ \exp\Big( \int_0^t \{T_y \xi_{\epsilon}(x + X^t_s) - c_{\epsilon}\} ds \Big)
    \indic_{\{x + X^t_{[0, t]} \subseteq Q_L\}} \Big] \\
    &\leq \frac{1}{2\pi t} \expect \Big[ \exp\Big( \int_0^t \{T_{y+z} \xi_{\epsilon}(x - z + X^t_s) - c_{\epsilon}\} ds \Big)
    \indic_{\{x - z + X^t_{[0, t]} \subseteq Q_{2L}\}} \Big] \\ &= u_{2L, \epsilon}^{x-z, y+z}(t, x-z).
  \end{align*}
  Hence $u_{\infty}^{x,y}(t, x) \leq u_{\infty}^{x-z, y+z}(t, x-z)$, and hence $u_{\infty}^{x,y}(t, x) = u_{\infty}^{x-z, y+z}(t, x-z)$ by symmetry.
\end{proof}
We next discuss relation between $u_{\infty}^{x,y}$ and the solution of the PAM in $\R^2$.
\begin{definition}
  We denote by $\tilde{u}^{\phi}_{\infty, \epsilon}$ the solution of the following PAM in $\R^2$:
  \begin{equation*}
    \begin{cases}
      \partial_t \tilde{u}^{\phi}_{\infty, \epsilon} = \frac{1}{2} \Delta \tilde{u}^{\phi}_{\infty, \epsilon} +
      (\xi_{\epsilon} - c_{\epsilon}) \tilde{u}^{\phi}_{\infty, \epsilon}, \,\, \mbox{in } (0, \infty) \times \R^2, \\
      \tilde{u}^{\phi}_{\infty, \epsilon}(0, \cdot) = \phi.
    \end{cases}
  \end{equation*}
  According to \cite{hairer_Labbe_2015}, if $\phi$ is a H\"older distribution of regularity better than $-1$ with
  at most exponential growth at infinity, the solution $\tilde{u}^{\phi}_{\infty, \epsilon}$ converges uniformly
  on compact sets of $(0, \infty) \times \R^2$ in probability and we denote the limit by $\tilde{u}^{\phi}_{\infty}$.
  Although this result cannot be applied when $\phi = \delta_x$, the work \cite{hairer_multiplicative_2018} constructed the solution
  of the three-dimensional PAM with initial condition $\delta_x$, and by adopting its approach we can prove that the solution
  $\tilde{u}^{\delta_x}_{\infty, \epsilon}$ converges in probability and we denote the limit by $\tilde{u}^{x}_{\infty}$.
\end{definition}
\begin{proposition}\label{prop:pam_in_plane}
  For each $x \in \R^2$, we have $u_{\infty}^{x,0} = \tilde{u}^x_{\infty}$ $\wnP$-almost surely.
\end{proposition}
\begin{proof}
By the Feynman-Kac formula \cite[Section II.3.2]{carmona1990spectral}, we have
\begin{equation}\label{eq:feynman_kac_for_u_infty}
  \tilde{u}^{\delta_x}_{\infty, \epsilon}(t, y)
  = p_t(x-y) \expect \Big[\exp\Big(\int_0^t \{\xi_{\epsilon}(y+B_s) - c_{\epsilon}\} ds \Big) \Big\vert B_t = x-y \Big].
\end{equation}
Therefore, by comparing \eqref{eq:feynman-kac} and \eqref{eq:feynman_kac_for_u_infty},
we obtain $u_{L, \epsilon}^{x, 0} \leq \tilde{u}_{\infty, \epsilon}^{\delta_x}$.
By letting $\epsilon \to 0+$ and then $L \to \infty$, we get
\begin{equation}\label{eq:compare_u_and_tilde_u}
  u_{\infty}^{x, 0} \leq \tilde{u}_{\infty}^x.
\end{equation}
Furthermore, due to the property of conditional expectation, \eqref{eq:feynman_kac_for_u_infty} implies
\begin{equation}\label{eq:int_of_tilde_u_infty}
  \int_{\R^2} \tilde{u}^{\delta_x}_{\infty, \epsilon}(t, y)dy
  = \expect \Big[\exp\Big(\int_0^t \{\xi_{\epsilon}(y+B_s) - c_{\epsilon}\} ds \Big)\Big]
  = \tilde{u}^{\indic}_{\infty, \epsilon}(t, x),
\end{equation}
where the second equality is again thanks to the Feynman-Kac formula.
If we set $U_L(t, x) \defby u_L^{\indic, \bm{\xi}_L^{\n}}(t, x)$, we similarly
have
\begin{equation}\label{eq:U_L_and_int_of_u_L}
  U_L(t, x) = \int_{\R^2} u_L^{x, 0}(t, y) dy,
\end{equation}
as given in \cite[(43)]{konig2020longtime}.

By \eqref{eq:U_L_and_int_of_u_L} and the monotone convergence theorem, we have
\begin{equation*}
  \lim_{L \to \infty} U_L(t, x) = \int_{\R^2} u_{\infty}^{x, 0}(t, y) dy.
\end{equation*}
By \eqref{eq:compare_u_and_tilde_u}, we have
\begin{equation}\label{eq:compare_integrals_of_u_and_tilde_u}
  \int_{\R^2} u_{\infty}^{x, 0}(t, y) dy \leq
  \int_{\R^2} \tilde{u}_{\infty}^x(t, y)dy.
\end{equation}
By Fatou's lemma and \eqref{eq:int_of_tilde_u_infty}, we have
\begin{equation*}
  \int_{\R^2} \tilde{u}_{\infty}^x(t, y) dy
  \leq \liminf_{\epsilon \to 0} \int_{\R^2} \tilde{u}^x_{\infty, \epsilon}(t, y)dy
  = \tilde{u}_{\infty}^{\indic}(t, x).
\end{equation*}
However, according to \cite[Lemma 4.3]{konig2020longtime}, we have
$\lim_{L \to \infty} U_L = \tilde{u}_{\infty}^{\indic}$ $\wnP$-a.s.
Therefore, the inequality in \eqref{eq:compare_integrals_of_u_and_tilde_u}
is in fact an equality and hence, combined with \eqref{eq:compare_u_and_tilde_u},
$u_{\infty}^{x, 0} = \tilde{u}_{\infty}^x$ $\wnP$-a.s.
\end{proof}
Finally, we review the Anderson Hamiltonian $\H_L^{\bm{\theta}}$ on $Q_L$.
\begin{definition}
  Let $\alpha \in (-\frac{4}{3}, -1)$, $L \in \mathbb{I}$ and $\bm{\theta} \in \X_{L}^{\n, \alpha}$.
  According to \cite[Theorem 5.4]{chouk2020asymptotics},
  we can construct the unbounded self-adjoint operator $\H_L^{\bm{\theta}}$ defined by
  \begin{equation*}
    \H_L^{\bm{\theta}} u \defby -\frac{1}{2} \Delta u - \bm{\theta} \diamond u,
    \quad u \in \operatorname{Dom}(\H_L^{\bm{\theta}})
  \end{equation*}
  The operator $\H_L^{\bm{\theta}}$ has eigenvalues $\lambda_{1, L}^{\bm{\theta}} \leq \lambda_{2, L}^{\bm{\theta}} \leq \cdots$
  (counting multiplicities) which diverge to infinity.
  We have an orthonormal basis $\{v_{n, L}^{\bm{\theta}}\}_{n=1}^{\infty}$ of $L^2(Q_L)$ such that
  $v_{n, L}^{\bm{\theta}}$ is an eigenvector of $\H_L^{\bm{\theta}}$ with eigenvalue $\lambda_{n, L}^{\bm{\theta}}$.
\end{definition}
By \cite[Theorem 2.12]{konig2020longtime}, we have
\begin{equation*}
  u_L^{\delta_x, \bm{\theta}}(t, \cdot) = \sum_{n=1}^{\infty} e^{-t\lambda^{\bm{\theta}}_{n, L}} v_{n,L}^{\bm{\theta}}(x)
  v_{n,L}^{\bm{\theta}}(\cdot) \quad \mbox{in } L^2(Q_L),
  \quad \mbox{for } x \in Q_L, t \in (0, \infty).
\end{equation*}
Thus, Fubini's theorem yields
\begin{equation}\label{eq:trace_of_exp_of_AH}
  \sum_{n=1}^{\infty} e^{-t \lambda_{n,L}^{\bm{\theta}}} =
  \sum_{n=1}^{\infty} e^{-t \lambda_{n,L}^{\bm{\theta}}} \int_{Q_L} v_{n, L}^{\bm{\theta}} (x)^2 dx
  = \int_{Q_L} u_L^{\delta_x, \bm{\theta}}(t, x) dx.
\end{equation}
\begin{definition} Let $L \in \mathbb{I}$.
  \begin{enumerate}[(i)]
    \item We denote by
    $\H_L \defby \H_L^{\bm{\xi}_L^{\n}}$ the Anderson Hamiltonian
    with white noise on $Q_L$.
    \item We set $\lambda_{n, L} \defby \lambda_{n, L}^{\bm{\xi}_L^{\n}}$. We introduce the eigenvalue counting function $N_L$ defined by
    \begin{equation}\label{eq:def_of_N_L}
      N_L(\lambda) \defby \frac{1}{L^2} \sum_{n=1}^{\infty} \indic_{\{\lambda_{n, L} \leq \lambda\}}.
    \end{equation}
    The function $N_L$ is nondecreasing and right-continuous, and thus naturally induces the Lebesgue-Stieltjes measure $N_L(d\lambda)$.
  \end{enumerate}
\end{definition}

\section{Self-intersection local time of Brownian bridges}\label{sec:brownian_bridge}
Here we study SILT of two-dimensional Brownian bridges, which is used to
construct the IDS of the Anderson Hamiltonian with white noise in Section \ref{sec:ids}.
The SILT $\gamma$ of the Brownian motion $B$ has been extensively studied, see \cite{chen_random_2010} and references therein.
It is formally given by
\begin{equation*}
  \gamma([0,t]_{\leq}^2) = \iint_{[0,t]_{\leq}^2} \delta_0( B_s - B_r) dr ds.
\end{equation*}
For rigorous construction, one naturally considers the limit $\lim_{\epsilon \to 0+} \iint_{[0,t]_{\leq}^2}
p_{\epsilon}(B_s - B_r) dr ds$. However, this limit does not converge. In fact, the construction of
$\gamma([0,t]_{\leq}^2)$ requires renormalization, and hence we have the following
definition.
\begin{definition}\label{def:def_of_gamma}
  We set
  \begin{equation*}
    \gamma([0,t]_{\leq}^2) \defby \lim_{\epsilon \to 0+}\iint_{[0,t]_{\leq}^2} \{p_{\epsilon}(B_s - B_r) - \expect[p_{\epsilon}(B_s - B_r)] \}dr ds,
  \end{equation*}
  where the convergence in $L^2(\P)$ is proved by \cite[Theorem 2.3.2 and Theorem 2.4.1]{chen_random_2010}.
  We note that the convergence actually takes place in $L^p(\P)$ for every $p \in (0, \infty)$ by \eqref{eq:uniform_integrability_of_e_to_beta}.
\end{definition}
By the scaling property of $B$, we have $\gamma([0,t]_{\leq}^2) \dequal t \gamma([0,1]_{\leq}^2)$.
According to \cite{bass_self-intersection_2004} or \cite[Theorem 4.3.1]{chen_random_2010}, we have
\begin{equation*}
  \lim_{t \to \infty} \frac{1}{t} \log \P(\gamma([0,1]_{\leq}^2) \geq t) = - \kappa^{-1},
\end{equation*}
where $\kappa$ is the best constant of Ladyzhenskaya's inequality \eqref{eq:GN_ineq}.

Recall from Notation that we denote by $X^t$ the Brownian bridge with
$X^t_0 = X^t_t = 0$.
The aim of this section is to prove this form of the large deviation where the Brownian motion is replaced by the Brownian bridge $X^1$.
The following lemma serves as an important technical tool.
Note that the process $(B_s - \frac{s}{t} (B_t - x))_{s \in [0, t]}$ has the same law as the Brownian bridge from $0$ at $s = 0$ to $x$ at $s = t$.
\begin{lemma}[{\cite[Lemma 3.1]{nakao_spectral_1977}}]\label{lem:girsanov_brownian_bridge}
  Let $x \in \R^d$ and $u \in (0, t)$. We define the probability measure $\tilde{\P}$ by
  \begin{equation*}
    d\tilde{\P} \defby \Big(\frac{t}{t-u}\Big)^{\frac{d}{2}}
    \exp\Big( - \frac{\abs{x}^2u}{2t(t-u)} - \frac{\abs{B_u}^2}{2(t-u)} + \frac{\inp{x}{B_u}}{t-u} \Big) d\P.
  \end{equation*}
  Then, the law of $(B_s)_{s \in [0, u]}$ under $\tilde{\P}$ is that of $(B_s - \frac{s}{t} (B_t - x))_{s \in [0, u]}$ under $\P$.
\end{lemma}
\begin{comment}
\begin{proof}
  The solution $Y$ of the SDE
  \begin{equation*}
    dY_s = \frac{x - Y_s}{t - s} ds + dB_s
  \end{equation*}
  has the law of the Brownian bridge from $0$ at $s=0$ to $x$ at $s=t$.
  Therefore, the claim follows by the Girsanov theorem.
  See \cite[Lemma 3.1]{nakao_spectral_1977} for detail.
\end{proof}
\end{comment}
\begin{definition}\label{def:def_of_chi_epsilon}
  For a Borel set $A \subseteq [0,t]_{\leq}^2$, we set
  \begin{equation*}
    \chi^t_{\epsilon}(A) \defby \iint_A p_{\epsilon}(X^t_s - X^t_r) dr ds
  \end{equation*}
  and
  \begin{equation*}
    \beta_{\epsilon} (A) \defby \iint_A p_{\epsilon}(B_s - B_r) dr ds.
  \end{equation*}
  In addition, we set
  \begin{equation*}
    \alpha_{\epsilon}(A) \defby \iint_A p_{\epsilon}(B_s - \tilde{B}_r) dr ds,
  \end{equation*}
  where $\tilde{B}$ is an independent copy of $B$.
\end{definition}
The SILT of $X^t$ should be given as the limit of $\chi^t_{\epsilon}([0,t]_{\leq}^2)
- \expect[\chi^t_{\epsilon}([0,t]_{\leq}^2)]$ as $\epsilon \to 0+$.
We first estimate the renormalization constant $\expect[\chi^t_{\epsilon}([0,t]_{\leq}^2)]$.
\begin{lemma}\label{lem:renormalization_constant_of_SILT_of_BB}
  For $t \in (0, \infty)$ and $0 \leq a \leq b \leq c \leq d \leq t$, we have, as $\epsilon \to 0+$,
  \begin{align*}
    \expect[\chi^t_{\epsilon}([a,b]_{\leq}^2)] &= \frac{1}{2\pi} \Big[  (b-a)\log \epsilon^{-1} + (b-a)  \log t \\
    &\hspace{2cm}+ \int_0^{b-a}\{ \log s - \log (t-s)\} ds \Big] + o(1), \\
    \expect[\chi^t_{\epsilon}([a, b] \times [c, d])]
    &= \frac{1}{2\pi} \int_{c-a}^{d-a} \{\log s - \log(t-s)\} ds  \\
    &\hspace{2cm}- \frac{1}{2\pi} \int_{c-b}^{d-b} \{\log s - \log(t-s)\} ds + o(1).
  \end{align*}
\end{lemma}
\begin{proof}
  Let $s \geq r$. Since $(X^t_s)_{s \in [0, t]} \dequal (B_s - \frac{s}{t} B_t)_{s \in [0, t]}$, $X^t_s - X^t_r$ is centered Gaussian with
  \begin{equation*}
    \expect[(X^t_s - X^t_r)^2]
    = s - r - \frac{(s-r)^2}{t}.
  \end{equation*}
  Therefore, we have
  \begin{align*}
    2 \pi \epsilon \expect[p_{\epsilon}(X_s^t - X_r^t)]
    &=  \int_{\R^2} \exp\Big( -\frac{1}{2\epsilon} \Big(s-r - \frac{(s-r)^2}{t}\Big) \abs{x}^2 \Big)
    \frac{e^{-\frac{\abs{x}^2}{2}}}{2\pi}  dx \\
    &= \Big( 1+ \frac{1}{\epsilon} \Big(s-r - \frac{(s-r)^2}{t} \Big) \Big)^{-1}
  \end{align*}
  and
  \begin{align*}
    2 \pi \iint_{[a,b]_{\leq}^2} \expect[p_{\epsilon}(X_s^t - X_r^t)] dr ds
    &= \int_a^b \int_a^s \frac{dr}{\epsilon + s-r - \frac{(s-r)^2}{t}} \, ds \\
    &= \int_a^b \int_0^{s-a} \frac{dr}{\epsilon + r - \frac{r^2}{t}} \, ds.
  \end{align*}
  We compute
  \begin{align*}
    \int_0^{s-a} \frac{dr}{\epsilon + r -\frac{r^2}{t}} dr
    &= \int_0^{s-a} \frac{dr}{\epsilon + \frac{t}{4} - \frac{1}{t}(r - \frac{t}{2})^2} \\
    &= \frac{t}{\sqrt{4\epsilon t + t^2}} \int_0^{s-a} \Big( \frac{1}{\sqrt{\epsilon t + \frac{t^2}{4}} + r - \frac{t}{2}}
    + \frac{1}{\sqrt{\epsilon t + \frac{t^2}{4}} - r + \frac{t}{2}} \Big) dr \\
    &= \frac{t}{\sqrt{4\epsilon t + t^2}}
    \Big\{ \log\Big( \sqrt{\epsilon t + \frac{t^2}{4}} - \frac{t}{2} + s - a \Big)
    - \log\Big( \sqrt{\epsilon t + \frac{t^2}{4}} - \frac{t}{2}  \Big) \\
    &\hspace{2cm} - \log\Big( \sqrt{\epsilon t + \frac{t^2}{4}} + \frac{t}{2} - s + a \Big)
    + \log\Big( \sqrt{\epsilon t + \frac{t^2}{4}} + \frac{t}{2}  \Big) \Big\}.
  \end{align*}
  Thus, $2 \pi \iint_{[a,b]_{\leq}^2} \expect[p_{\epsilon}(X_s^t - X_r^t)] dr ds$ equals to
  \begin{align*}
    &\frac{t}{\sqrt{4\epsilon t + t^2}}
    \Big[ (b-a) \log\Big( \sqrt{\epsilon t + \frac{t^2}{4}} + \frac{t}{2} \Big)
    - (b-a) \log\Big( \sqrt{\epsilon t + \frac{t^2}{4}} - \frac{t}{2}  \Big) \\
    &\hspace{2cm} + \int_0^{b-a} \Big\{ \log\Big( \sqrt{\epsilon t + \frac{t^2}{4}} - \frac{t}{2} + s \Big)
    -  \log\Big( \sqrt{\epsilon t + \frac{t^2}{4}} + \frac{t}{2} -s  \Big) \Big\} ds \Big]
  \end{align*}
  To estimate $\expect[\chi^t_{\epsilon}([a,b]_{\leq}^2)]$, it remains to compute
  \begin{align*}
    \log \Big( \sqrt{\epsilon t + \frac{t^2}{4}} - \frac{t}{2}  \Big)
    &= \log \frac{t}{2} + \log \Big( \sqrt{1 + \frac{4\epsilon}{t}} - 1 \Big) \\
    &= \log \frac{t}{2} + \log \Big( \frac{2\epsilon}{t} + o(\epsilon) \Big) = \log \epsilon + o(1).
  \end{align*}

  Similarly, we compute
  \begin{align*}
    \MoveEqLeft[3]
    \lim_{\epsilon \to 0+} 2\pi \expect[\chi^t_{\epsilon}([a,b]\times [c,d]) ] \\
    &= \int_c^d \int_{s-b}^{s-a} \Big\{\frac{1}{r} + \frac{1}{t-r} \Big\} dr ds \\
    &= \int_c^d \{\log(s-a) - \log(s-b) - \log(t-s +a) + \log(t-s+b)\} ds. \qedhere
  \end{align*}
\end{proof}
\begin{remark}\label{rem:renormalization_constant_of_SILT_of_BM}
  Similar calculation yields
  \begin{equation*}
    \expect[\beta_{\epsilon}([a, b]_{\leq}^2)]
    = \frac{b-a}{2\pi} \{\log \epsilon^{-1} + \log (b-a) - 1\} + o(1).
  \end{equation*}
\end{remark}

\begin{definition}\leavevmode
  \begin{enumerate}[(i)]
    \item Let $A$ be a Borel set of $[0, \infty)^2$.
     It is known that $\alpha_{\epsilon}(A)$ converges in $L^p(\P)$ for every $p \in (0, \infty)$ \cite[Theorem 2.2.3]{chen_random_2010}.
    We set
    \begin{equation*}
      \alpha(A) \defby \lim_{\epsilon \to 0+} \alpha_{\epsilon}(A).
    \end{equation*}
    \item
    As mentioned in Definition \ref{def:def_of_gamma},
    $\beta_{\epsilon}([a,b]_{\leq}^2) - \expect[\beta_{\epsilon}([a,b]_{\leq}^2)]$ converges
    to $\gamma([a,b]_{\leq}^2)$
    in $L^p(\P)$ for every $p \in (0, \infty)$.
    In contrast, if $A$ is of the form
    \begin{equation}\label{eq:sum_of_nondiagonal_squares}
      \cup_{k=1}^n [a_i, b_i] \times [c_i, d_i] \quad \mbox{with } 0 \leq
    a_i \leq b_i \leq c_i \leq d_i,
    \end{equation}
    the random variable
    $\beta_{\epsilon}(A)$ converges
    in $L^p(\P)$ for every $p \in (0, \infty)$ without renormalization
    \cite[Theorem 2.3.2]{chen_random_2010}.
    Thus, for such $A$, we set
    \begin{equation*}
      \beta(A) \defby \lim_{\epsilon \to 0+} \beta_{\epsilon}(A).
    \end{equation*}
  \end{enumerate}
\end{definition}

We have the scaling property: $\alpha([0, t]^2) \dequal t \alpha([0,1]^2)$ \cite[Proposition 2.2.6]{chen_random_2010}.
Later, we use the following inequality
\begin{equation}\label{eq:le_gall_moment_formula}
  \expect[\alpha(A_1 \times A_2)^m ] \leq \expect[\alpha(A_1^2)^m]^{\frac{1}{2}} \expect[\alpha(A_2^2)^m]^{\frac{1}{2}},
  \quad m \in \N, \,\, A_1, A_2 \subseteq [0, \infty)
\end{equation}
from \cite[(2.2.12)]{chen_random_2010}.
According to \cite[Theorem 3.3.2]{chen_random_2010}, $\alpha$ satisfies the following large deviation
\begin{equation}\label{eq:large_deviation_of_alpha}
  \lim_{t \to \infty} \frac{1}{t} \log \P(\alpha([0, 1]^2) \geq t) = - \kappa^{-1}.
\end{equation}

As indicated in \cite[Proposition 2.3.4]{chen_random_2010}, independence of increments of the Brownian motion implies
\begin{equation}\label{eq:alpha_and_beta}
  \beta_{\epsilon}([a, b] \times [b, d]) \dequal \alpha_{\epsilon}([0, b-a] \times [0, d - b]), \quad 0 \leq a \leq b \leq d.
\end{equation}
Besides, as shown in the proof of \cite[(4.3.8)]{chen_random_2010},
for every $\lambda \in (0, \infty)$,
\begin{equation}\label{eq:uniform_integrability_of_e_to_alpha}
  \sup_{\epsilon \in (0, \infty)} \expect[ e^{\lambda \alpha_{\epsilon}(A)} ]
  \leq \expect[ e^{\lambda \alpha(A)} ], \qquad
  \sup_{\epsilon \in (0, \infty)} \expect[ e^{\lambda \beta_{\epsilon}(B)} ]
  \leq \expect[ e^{\lambda \beta(B)} ],
\end{equation}
where $B$ is of the form \eqref{eq:sum_of_nondiagonal_squares}

The work \cite{varadhan_quantum_1969} constructed the SILT of the Brownian bridge $X^t$.
We give an alternative proof by taking advantage of the construction of $\gamma$.
\begin{theorem}\label{thm:construction_of_SILT_of_BB}
  Let $t \in (0, \infty)$.
  \begin{enumerate}[(i)]
    \item For $0 \leq a \leq b \leq c \leq d \leq t$ and $p \in (0, \infty)$, the net
    $\{\chi^t_{\epsilon}([a,b] \times [c,d])\}_{\epsilon > 0}$ converges in $L^p(\P)$ as $\epsilon \to 0+$.
    \item For $0 \leq a \leq b \leq t$ and $p \in (0, \infty)$, the net
    $\{ \chi^t_{\epsilon}([a,b]_{\leq}^2) - \expect[\chi^t_{\epsilon}([a,b]_{\leq}^2)]\}_{\epsilon > 0}$
    converges in $L^p(\P)$ as $\epsilon \to 0+$.
    \item For every $\lambda \in (0, \kappa^{-1})$,
      \begin{equation*}
        \sup_{\epsilon \in (0, \infty)} \expect\Big[ \exp \Big\{\lambda (\chi_{\epsilon}^1([0,1]_{\leq}^2) -
        \expect[\chi_{\epsilon}^1([0,1]_{\leq}^2)]) \Big\}\Big] < \infty.
      \end{equation*}
  \end{enumerate}
\end{theorem}
\begin{proof}
  (i) Suppose $d < t$. Then, by Lemma \ref{lem:girsanov_brownian_bridge},
  \begin{multline*}
    \expect[ \abs{\chi^t_{\epsilon}([a,b]\times [c,d]) - \chi^t_{\epsilon'}([a,b]\times [c,d])}^p]\\
    \leq \frac{t}{t-d} \expect[ \abs{\beta_{\epsilon}([a,b]\times [c,d]) - \beta_{\epsilon'}([a,b]\times [c,d])}^p]
    \to 0 \quad (\epsilon, \epsilon' \to 0+).
  \end{multline*}
  Therefore, the claim for the case $d < t$ is proved. Thank to the reversibility of $X^t$, the case $a > 0$ and $d = t$ is covered as well.
  It remains to consider the case $a = 0$ and $d = t$.
  Let $\delta \in (0, \min\{b, t-c, \frac{t}{4}\})$. We decompose $\chi_{\epsilon}^t([0,b] \times [c, t])$ into the sum
  \begin{equation*}
    \chi_{\epsilon}^t([0, b] \times [c, t - \delta])
    + \chi_{\epsilon}^t([\delta, b] \times [t - \delta, t])
    + \chi_{\epsilon}^t([0, \delta] \times [t - \delta, t]).
  \end{equation*}
  The first two terms converges in $L^p(\P)$. Thus, we will prove the convergence of the third term.
  The key is to observe that, under the condition that $X_{t/2}^t = x$,
  the processes $(X^t_s)_{s \in [0, \frac{t}{2}]}$ and $(X_{t-s}^t)_{s \in [0, \frac{t}{2}]}$ are
  independent and identically distributed, and
  the distribution is the Brownian bridge $Y$ from $0$ at $s=0$ to $x$ at $s=\frac{t}{2}$.
  Let $\tilde{Y}$ and $\tilde{B}$ be independent copies of $Y$ and $B$ respectively.
  Then, since $X^t_{t/2}$ is a centered Gaussian with variance $\frac{t}{4}$,
  the preceding remark implies
  \begin{multline*}
    \expect[ \abs{\chi^t_{\epsilon}([0, \delta] \times [t-\delta, t]) - \chi^t_{\epsilon'}([0, \delta] \times [t-\delta, t])}^p ] \\
    = \int_{\R^2} dx p_{\frac{t}{4}}(x) \expect\Big[ \Big\lvert \iint_{[0, \delta]^2}
    \{ p_{\epsilon}(Y_s - \tilde{Y}_r) - p_{\epsilon'}(Y_s - \tilde{Y}_r) \} \Big\rvert^p \Big]
  \end{multline*}
  By Lemma \ref{lem:girsanov_brownian_bridge}, we have
  \begin{multline*}
    \expect\Big[ \Big\lvert \iint_{[0, \delta]^2}
    \{ p_{\epsilon}(Y_s - \tilde{Y}_r) - p_{\epsilon'}(Y_s - \tilde{Y}_r) \} \Big\rvert^p \Big] \\
    = \frac{t}{t - 2 \delta}
      \expect\Big[ \exp\Big( -\frac{4 \delta \abs{x}^2}{t(t-2\delta)} -\frac{\abs{B_{\delta}}^2 + \abs{\tilde{B}_{\delta}}^2}
      {t-2\delta} + \frac{2 \inp{x}{B_{\delta} + \tilde{B}_{\delta}}}{t-2\delta} \Big) \\
      \times \abs{\alpha_{\epsilon}([0, \delta]^2) - \alpha_{\epsilon'}([0, \delta]^2)}^p \Big].
  \end{multline*}
  Therefore, by applying the Cauchy-Schwarz inequality, we obtain
  \begin{multline*}
    \expect[ \abs{\chi^t_{\epsilon}([0, \delta] \times [t-\delta, t]) - \chi^t_{\epsilon'}([0, \delta] \times [t-\delta, t])}^p ] \\
    \leq \frac{t}{t-2 \delta} \expect[
    \abs{\alpha_{\epsilon}([0, \delta]^2) - \alpha_{\epsilon'}([0, \delta]^2)}^{2p}]^{\frac{1}{2}} \\
    \times \int_{\R^2} dx p_{\frac{t}{4}}(x)
    e^{-\frac{4 \delta \abs{x}^2}{t(t - 2 \delta)}}
    \expect\Big[ \exp \Big(-\frac{2 \abs{B_{\delta}}^2}{t - 2 \delta}
    + \frac{4 \inp{x}{B_{\delta}}}{t - 2 \delta} \Big)\Big].
  \end{multline*}
  We compute
  \begin{align*}
    \expect\Big[ \exp \Big(-\frac{2 \abs{B_{\delta}}^2}{t - 2 \delta}
    + \frac{4 \inp{x}{B_{\delta}}}{t - 2 \delta} \Big)\Big]
    &= \int_{\R^2} \exp \Big(-\frac{2 \delta \abs{u}^2}{t - 2 \delta}
    + \frac{4 \sqrt{\delta} \inp{x}{u}}{t - 2 \delta} \Big)
    \frac{e^{-\frac{\abs{u}^2}{2}}}{2 \pi} du \\
    &= \frac{t - 2 \delta}{t} \exp\Big(
    \frac{8 \delta \abs{x}^2}{t(t - 2 \delta)} \Big)
  \end{align*}
  and observe
  \begin{equation*}
    \int_{\R^2}  p_{\frac{t}{4}}(x)
    e^{-\frac{4 \delta \abs{x}^2}{t(t - 2 \delta)}}
    e^{\frac{8 \delta \abs{x}^2}{t(t - 2 \delta)}} dx
    = \frac{2}{\pi t}
    \int_{\R^2} e^{-\frac{2 \abs{x}^2}{t}}
    e^{\frac{4 \delta \abs{x}^2}{t(t - 2 \delta)}} dx
  \end{equation*}
  is finite, since $\delta < \frac{t}{4}$.
  As $\{\alpha_{\epsilon}([0, \delta]^2)\}_{\epsilon > 0}$ converges in
  $L^{2p}(\P)$, these estimates show
  \begin{equation*}
    \lim_{\epsilon, \epsilon' \to 0+}
    \expect[ \abs{\chi^t_{\epsilon}([0, \delta] \times [t-\delta, t]) - \chi^t_{\epsilon'}([0, \delta] \times [t-\delta, t])}^p ] = 0
  \end{equation*}
  and hence we end the proof of (i).

  (ii) Set $c \defby \frac{a+b}{2}$. We decompose
  \begin{equation*}
    \chi^t_{\epsilon}([a,b]_{\leq}^2) - \expect[\chi^t_{\epsilon}([a,b]_{\leq}^2)]
    = \sum_{j=1}^3 \{\chi^t_{\epsilon}(A_j) - \expect[\chi^t_{\epsilon}(A_j)]\},
  \end{equation*}
  where $A_1 \defby [a, c]^2_{\leq}$, $A_2 \defby [c, b]^2_{\leq}$ and $A_3 \defby [a, c] \times [c, b]$.
  We already know $\chi^t_{\epsilon}(A_3)$ converges.
  We will consider the convergence of $\chi^t_{\epsilon}(A_1) - \expect[\chi^t_{\epsilon}(A_1)]$,
  as the convergence of $\chi^t_{\epsilon}(A_2) - \expect[\chi^t_{\epsilon}(A_2)]$ can be proved similarly thanks to
  the reversibility of $X^t$.
  By Lemma \ref{lem:renormalization_constant_of_SILT_of_BB} and Remark \ref{rem:renormalization_constant_of_SILT_of_BM}, we have
  \begin{equation*}
    \big\{ \expect[\chi_{\epsilon}^t(A_1)] - \expect[\chi_{\epsilon'}^t(A_1)] \big\}
    - \big\{ \expect[\beta_{\epsilon}(A_1)] - \expect[\beta_{\epsilon'}(A_1)] \big\} = o(1)
  \end{equation*}
  as $\epsilon, \epsilon' \to 0$.
  Therefore, combined with Lemma \ref{lem:girsanov_brownian_bridge}, we obtain
  \begin{multline*}
    \expect\big[ \abs*{\{\chi^t_{\epsilon}(A_1) - \expect[\chi^t_{\epsilon}(A_1)]\} -
    \{\chi^t_{\epsilon'}(A_1) - \expect[\chi^t_{\epsilon'}(A_1)]\} }^p \big] \\
    \leq \frac{\max\{2^{p-1}, 1\} t}{t-c} \expect\big[ \abs*{\{\beta_{\epsilon}(A_1) - \expect[\beta_{\epsilon}(A_1)]\} -
    \{\beta_{\epsilon'}(A_1) - \expect[\beta_{\epsilon'}(A_1)]\} }^p \big] + o(1) = o(1),
  \end{multline*}
  which complete the proof of (ii).

  (iii) Let $\mu \in (0, \infty)$ and $\delta \in (0, \frac{1}{4})$. We decompose
  \begin{equation}\label{eq:decomposition_for_chi}
    \chi^1_{\epsilon}([0,1]_{\leq}^2) - \expect[\chi^1_{\epsilon}([0,1]_{\leq}^2)]
    = \sum_{j=1}^4 \{\chi^1_{\epsilon}(A_j) - \expect[\chi^1_{\epsilon}(A_j)]\},
  \end{equation}
  where
  \begin{align*}
    &A_1 \defby [0, 1-\delta]_{\leq}^2, \,\, A_2 \defby [1-\delta, 1]_{\leq}^2, \\
    &A_3 \defby [\delta, 1-\delta] \times [1-\delta, 1], \,\, A_4 \defby [0, \delta] \times [1-\delta, 1].
  \end{align*}
  By Lemma \ref{lem:renormalization_constant_of_SILT_of_BB} and Remark \ref{rem:renormalization_constant_of_SILT_of_BM},
  \begin{equation*}
    \chi_{\epsilon}^1(A_1) - \expect[ \chi^1_{\epsilon}(A_1)]
    = \chi_{\epsilon}^1(A_1) - \expect[ \beta_{\epsilon}(A_1)]  + O(1).
  \end{equation*}
  Therefore, by Lemma \ref{lem:girsanov_brownian_bridge},
  \begin{equation*}
    \sup_{\epsilon \in (0, \infty)} \expect[ e^{\mu (\chi_{\epsilon}^1(A_1) - \expect[ \chi^1_{\epsilon}(A_1)] ) }]
    \lesssim_{\delta} \sup_{\epsilon \in (0, \infty)} \expect[e^{\mu (\beta_{\epsilon}(A_1) - \expect[\beta_{\epsilon}(A_1)])}].
  \end{equation*}
  We have a similar estimate of $\chi_{\epsilon}^1(A_2)$. To evaluate $\chi_{\epsilon}^1(A_3)$, we set
  \begin{equation*}
    \tilde{A}_3 \defby [\delta, 1 - \delta] \times [0, \delta],
    \quad
    \hat{A}_3 \defby [0, \delta] \times [0, 1 - 2 \delta].
  \end{equation*}
  We apply Lemma \ref{lem:girsanov_brownian_bridge}, \eqref{eq:alpha_and_beta} and \eqref{eq:uniform_integrability_of_e_to_alpha}
  to obtain
  \begin{equation*}
    \sup_{\epsilon \in (0, \infty)} \expect[e^{\mu \chi_{\epsilon}^1(A_3)}]
    \leq \delta^{-1} \sup_{\epsilon \in (0, \infty)} \expect[e^{\mu \beta_{\epsilon}(\tilde{A}_3)}]
    = \delta^{-1} \sup_{\epsilon \in (0, \infty)} \expect[e^{\mu \alpha_{\epsilon}(\hat{A}_3)}]
    \leq \delta^{-1}\expect[e^{\mu \alpha(\hat{A}_3)}] .
  \end{equation*}
  Using \eqref{eq:le_gall_moment_formula} and the scaling property of $\alpha$, we obtain
  \begin{align*}
    \expect[e^{\mu \alpha(A_3)}]
    &= \sum_{m=0}^{\infty} \frac{\mu^m}{m!} \expect[\alpha([0, \delta] \times [0, 1-2\delta])^m] \\
    &\leq \sum_{m=0}^{\infty} \frac{\mu^m}{m!} \expect[\alpha([0, \delta]^2)^m]^{\frac{1}{2}}
    \expect[\alpha([0, 1-2\delta])^m]^{\frac{1}{2}} \\
    &\leq \sum_{m=0}^{\infty} \frac{(\sqrt{\delta} \mu)^m}{m!} \expect[\alpha([0, 1]^2)^m]
    = \expect[e^{\sqrt{\delta} \mu \alpha([0,1]^2)} ].
  \end{align*}
  In particular, we note that $\expect[\chi^1_{\epsilon}(A_3)]$ is bounded.
  Finally, evaluation of $\chi^1_{\epsilon}(A_4)$ is similar to that of $\chi^t_{\epsilon}([0, \delta] \times [t-\delta, t])$ in
  (i). Indeed,
  \begin{align*}
    &\expect[ e^{\mu \chi^1_{\epsilon}(A_4)} ] \\
      &= \int_{\R^2} dx p_{\frac{1}{4}}(x) \frac{1}{1 - 2 \delta}
      \expect\Big[ \exp\Big( -\frac{4 \delta \abs{x}^2}{1-2\delta} -\frac{\abs{B_{\delta}}^2 + \abs{\tilde{B}_{\delta}}^2}
      {1-2\delta} + \frac{2 \inp{x}{B_{\delta} + \tilde{B}_{\delta}}}{1-2\delta} \Big)
       e^{\mu \alpha_{\epsilon}([0, \delta]^2)}  \Big] \\
    &\lesssim_{\delta}  \expect[ e^{2\mu \alpha_{\epsilon}([0, \delta]^2)} ]^{\frac{1}{2}}
    \leq \expect[ e^{2 \delta \mu \alpha([0, 1]^2)} ]^{\frac{1}{2}}.
  \end{align*}

  Suppose $\lambda \in (0, \kappa^{-1})$ and $p, q \in (1, \infty)$ satisfy $\lambda p < \kappa^{-1}$ and $p^{-1} + q^{-1} = 1$.
  In view of the decomposition \eqref{eq:decomposition_for_chi},
  H\"older's inequality yields
  \begin{multline*}
    \expect\Big[ \exp \Big\{\lambda (\chi_{\epsilon}^1([0,1]_{\leq}^2) -
        \expect[\chi_{\epsilon}^1([0,1]_{\leq}^2)]) \Big\}\Big] \\
    \leq \expect\Big[ \exp \Big\{p \lambda (\chi_{\epsilon}(A_1) -
    \expect[\chi_{\epsilon}(A_1)]) \Big\}\Big]^{\frac{1}{p}}
    \prod_{j=2}^4 \expect\Big[ \exp \Big\{3q \lambda  (\chi_{ \epsilon}(A_j) -
    \expect[\chi_{\epsilon}(A_j)]) \Big\}\Big]^{\frac{1}{3q}}.
  \end{multline*}
  We set
  \begin{equation*}
    \bar{\beta}_{\epsilon}(A) \defby
    \beta_{\epsilon}(A) - \expect[\beta_{\epsilon}(A)].
  \end{equation*}
  By the above computations, the right hand side is bounded by
  \begin{equation*}
    C_{\delta, p} \expect[e^{p \lambda \bar{\beta}_{\epsilon}(A_1)}]^{\frac{1}{p}}
    \expect[e^{3q \lambda \bar{\beta}_{\epsilon}(A_2)}]^{\frac{1}{3q}}
    \expect[e^{3q \sqrt{\delta} \lambda \alpha([0, 1]^2)}]^{\frac{1}{3q}}
    \expect[e^{6q \delta \lambda \alpha([0, 1]^2)}]^{\frac{1}{3q}}
  \end{equation*}
  uniformly over $\epsilon \in (0, \infty)$.
  By the scaling property of the Brownian motion, we have
  \begin{equation*}
    \expect[e^{p \lambda \bar{\beta}_{\epsilon}(A_1)}]
    = \expect[e^{p(1 - \delta) \lambda \bar{\beta}_{\delta^{-1} \epsilon}([0, 1]_{\leq}^2)}],
    \quad
    \expect[e^{3 q \lambda \bar{\beta}_{\epsilon}(A_2)}]
    = \expect[e^{3q  \delta \lambda \bar{\beta}_{\delta^{-1} \epsilon}([0, 1]_{\leq}^2)}].
  \end{equation*}
  By making $\delta$ sufficiently small and taking \eqref{eq:large_deviation_of_alpha} into account,
  it suffices to show
  \begin{equation}\label{eq:uniform_integrability_of_e_to_beta}
    \sup_{\epsilon \in (0, \infty)} \expect[e^{\lambda \bar{\beta}_{\epsilon}([0,1]_{\leq}^2)}] < \infty
  \end{equation}
  for every $\lambda \in (0, \kappa^{-1})$.
  According to \cite[Lemma A.1]{gu_moments_2018}, there exists $\lambda_0 \in (0, \kappa^{-1})$ such that
  \eqref{eq:uniform_integrability_of_e_to_beta} holds for $\lambda \in (0, \lambda_0)$.
  Let $\lambda \in (0, \kappa^{-1})$ and $p, q \in (1, \infty)$ satisfy $p^2 \lambda < \kappa^{-1}$ and $p^{-1} + q^{-1} = 1$.
  Take large $n \in \N$, which will be determined later.
  \begin{figure}
    \centering
    \begin{tikzpicture}[scale=0.5]
      \fill[blue!30!white] (0,8)--(0,1)--(1,1)--(1,2)--(2,2)--(2,3)--(3,3)--(3,4)--(4,4)--(4,5)--(5,5)--(5,6)--(6,6)--(6,7)--(7,7)--(7,8)--(8,8);
      \foreach \x in {0,1,...,7}
        \fill[green!30!white] (\x,\x)--($(\x,\x)+(0,1)$)--($(\x,\x)+(1,1)$);
      \node at (2, 6) [rectangle, fill=blue!30!white] {\color{blue} $A$};
      \node [below] at (3,2.5) {\color{green!50!black} $B_k$};
    \end{tikzpicture}
    \caption{$A$ and $B_k$ for $n=3$.} \label{fig:A_and_B_k}
  \end{figure}
  We decompose
  \begin{equation*}
    \bar{\beta}_{\epsilon}([0,1]_{\leq}^2)
    = \bar{\beta}_{\epsilon}(A)
    + \sum_{k=1}^{2^n} \bar{\beta}_{\epsilon}(B_k),
  \end{equation*}
  where (see Figure \ref{fig:A_and_B_k})
  \begin{align*}
    A &\defby \bigcup_{m=1}^n \bigcup_{k=1}^{2^{m-1}} \Big[ \frac{2k-1}{2^m}, \frac{2k}{2^m} \Big]^2, \\
    B_k &\defby \Big[\frac{k-1}{2^n}, \frac{k}{2^n}\Big]^2_{\leq} \quad (k=1, 2, \ldots, 2^n).
  \end{align*}
  Then, since $\{\beta_{\epsilon}(B_k)\}_{k=1}^{2^n}$ are i.i.d., H\"older's inequality yields
  \begin{equation*}
    \expect[e^{\lambda\bar{\beta}_{\epsilon}([0,1]_{\leq}^2)}]
    \leq \expect[e^{p\lambda\bar{\beta}_{\epsilon}(A)}]^{\frac{1}{p}}
    \expect[e^{2^{-n} q\lambda\bar{\beta}_{\epsilon}([0,1]_{\leq}^2) }]^{\frac{2^n}{q}}.
  \end{equation*}
  Since $\lim_{\epsilon \to 0+} \expect[\beta_{\epsilon}(A)] = \expect[\beta(A)] < \infty$ and
  $\expect[e^{p \lambda \beta_{\epsilon}(A)}] \leq \expect[e^{p \lambda \beta(A)}]$
  by \eqref{eq:uniform_integrability_of_e_to_alpha},
  if $2^n > q\lambda \lambda_0^{-1}$,
  \eqref{eq:uniform_integrability_of_e_to_beta} follows once we prove $\expect[e^{p \lambda \beta (A)}] < \infty$.
  However, since $\beta(A) = \expect[\beta(A)] + \gamma([0,1]_{\leq}^2) - \sum_{k=1}^{2^n} \gamma(B_k)$, we obtain
  \begin{equation*}
    \expect[e^{p \lambda  \beta(A)}] \leq e^{p \lambda  \expect[\beta(A)]}
    \expect[e^{p^2 \lambda  \gamma([0,1]_{\leq}^2)}]^{\frac{1}{p}} \expect[e^{-2^{-n} p q \lambda \gamma([0,1]_{\leq}^2)}]^{\frac{2^n}{q}}.
  \end{equation*}
  It remains to observe $\expect[e^{-2^{-n} q \lambda \gamma([0,1]_{\leq}^2)}] < \infty$ by \cite[Theorem 4.3.2]{chen_random_2010}.
\end{proof}
\begin{remark}\label{rem:uniform_integrability_of_e_to_alpha_bridge}
  As with the proof of Theorem \ref{thm:construction_of_SILT_of_BB}-(iii), we can prove
  \begin{equation*}
    \sup_{\epsilon \in (0, \infty)} \expect \Big[ \exp\Big( \lambda \int_0^t\int_0^t p_{\epsilon}(X_s^{t,1} - X_r^{t,2}) dr ds \Big) \Big]
    < \infty
  \end{equation*}
  for $\lambda \in (0, \kappa^{-1})$, where $\{X^{t,j}\}_{j=1, 2}$ are i.i.d. copies of $X^t$.
\end{remark}
\begin{definition}
  For $0 \leq a \leq b \leq c \leq d \leq t$,
  we set
  \begin{align*}
    \chi^t([a,b] \times [c,d]) &\defby \lim_{\epsilon \to 0+} \chi^t_{\epsilon}([a,b] \times [c, d]), \\
    \zeta^t([a,b]_{\leq}^2) &\defby \lim_{\epsilon \to 0+} \chi^t_{\epsilon}([a,b]_{\leq}^2) - \expect[\chi^t_{\epsilon}([a,b]_{\leq}^2)]
  \end{align*}
  and $\zeta \defby \zeta^1([0,1]_{\leq}^2)$.
\end{definition}

By the scaling property of $X^t$, we have $\zeta^t([0,t]_{\leq}^2) \dequal t \zeta^1([0,1]_{\leq}^2)$.
 By Theorem \ref{thm:construction_of_SILT_of_BB}-(iii), we have
$\expect[e^{\lambda \zeta}] < \infty$ for $\lambda \in (0, \kappa^{-1})$.
In the next theorem, we show that this is critical.
\begin{theorem}\label{thm:right_tail_of_SILT_of_BB}
  We have
  \begin{equation*}
    \lim_{t \to \infty} \frac{1}{t} \log \P(\zeta \geq t) = -\kappa^{-1}.
  \end{equation*}
\end{theorem}
\begin{proof}
  The proof is motivated by \cite[Theorem 4.3.1]{chen_random_2010}.
  Since
  \begin{equation*}
    \P(\zeta^1([0,1]_{\leq}^2) \geq n) = \P(\zeta^n([0,n]_{\leq}^2) \geq n^2),
  \end{equation*}
   it suffices to prove
  \begin{equation}\label{eq:right_tail_of_zeta_n}
    \lim_{n \to \infty,\, n \in \N} \frac{1}{n} \log \P(\zeta^n([0,n]_{\leq}^2) \geq n^2) = - \kappa^{-1}.
  \end{equation}
  We have
  \begin{equation}\label{eq:decomposition_of_zeta_n}
    \zeta^n([0,n]_{\leq}^2) = \sum_{k=1}^n \zeta^n([k-1,k]_{\leq}^2)
    + \sum_{k=1}^{n-1}\{ \chi^n([0,k] \times [k, k+1]) - \expect[\chi^n([0,k]\times [k, k+1])] \}.
  \end{equation}
  By Lemma \ref{lem:renormalization_constant_of_SILT_of_BB},
  \begin{align*}
    \MoveEqLeft[3]
    2 \pi \sum_{k=1}^{n-1} \expect[ \chi^n([0, k] \times [k, k+1]) ] \\
    &= \sum_{k=1}^{n-1} \Big\{ \int_k^{k+1} \log s ds - \int_{n-k-1}^{n-k} \log s ds
    - \int_0^1 \log s ds + \int_{n-1}^n \log s ds \Big\}\\
    &= \int_1^n \log s ds - \int_0^{n-1} \log s ds
    - (n-1) \int_0^1 \log s ds + (n-1)\int_{n-1}^n \log s ds \\
    &=  n \int_{n-1}^n \log s ds - n \int_0^1 \log s ds
  \end{align*}
  and hence
  \begin{equation}\label{eq:sum_of_expect_of_chi}
    \sum_{k=1}^{n-1} \expect[ \chi^n([0, k] \times [k, k+1]) ] = o(n^2)
  \end{equation}
  For every $\delta > 0$ and $n \geq 3$, by setting $\tilde{n} \defby \lfloor \frac{n}{2} \rfloor + 1$,
  \begin{align*}
  \P\Big( \Big \lvert\sum_{k=1}^n \zeta^n([k-1,k]_{\leq}^2) \Big\rvert \geq \delta n^2\Big)
  &\leq 2 \P\Big( \Big \lvert\sum_{k=1}^{\tilde{n}} \zeta^n([k-1,k]_{\leq}^2) \Big\rvert \geq \frac{\delta n^2}{2} \Big) \\
  &\leq 8 \P\Big( \Big \lvert\sum_{k=1}^{\tilde{n}} \gamma([k-1,k]_{\leq}^2) \Big\rvert \geq \frac{\delta n^2}{2} + o(n^2) \Big) \\
  &\leq 8 e^{-\frac{\lambda \delta n^2}{2} + o(n^2)} \expect\Big[ \exp\Big( \lambda \sum_{k=1}^{\tilde{n}}
  \abs{\gamma([k-1,k]_{\leq}^2)} \Big) \Big],
  \end{align*}
  where $\lambda \in (0, \kappa^{-1})$ so that $\expect[e^{\lambda \abs{\gamma([0,1]_{\leq}^2)}}] < \infty$.
  In the first inequality, we applied the reversibility of $X^n$. In the second inequality, we applied Lemma \ref{lem:girsanov_brownian_bridge}
  and the estimate
  \begin{equation*}
    \sum_{k=1}^{\tilde{n}} \{ \expect[\chi_{\epsilon}^n([k-1,k]_{\leq}^2)] - \expect[\beta_{\epsilon}([k-1,k]_{\leq}^2)] \}
    = O(n \log n),
  \end{equation*}
  which follows from Lemma \ref{lem:renormalization_constant_of_SILT_of_BB} and Remark \ref{rem:renormalization_constant_of_SILT_of_BM}.
  In the third inequality, we applied Chebyshev's inequality.
  Since $\{\gamma([k-1,k]_{\leq}^2)\}_{k=1}^n$ are independent and identically distributed,
  \begin{equation*}
    \expect\Big[ \exp\Big( \lambda \sum_{k=1}^{\tilde{n}} \gamma([k-1,k]_{\leq}^2) \Big) \Big]
    = \expect[ e^{\lambda \gamma([0,1]_{\leq}^2)} ]^{\tilde{n}}.
  \end{equation*}
  Hence, we obtain
  \begin{equation}\label{eq:tail_of_sum_of_zeta}
    \lim_{n \to \infty} \frac{1}{n} \log \P\Big( \Big \lvert\sum_{k=1}^n \zeta^n([k-1,k]_{\leq}^2) \Big\rvert \geq \delta n^2\Big)
    = - \infty \quad \mbox{for every } \delta > 0.
  \end{equation}
  To prove \eqref{eq:right_tail_of_zeta_n}, by \eqref{eq:decomposition_of_zeta_n},
  \eqref{eq:sum_of_expect_of_chi} and \eqref{eq:tail_of_sum_of_zeta},
  it suffices to show
  \begin{equation*}
    \lim_{n \to \infty} \frac{1}{n}
    \log \P \Big( \sum_{k=1}^{n-1} \chi^n([0, k] \times [k, k+1]) \geq \rho n^2 \Big) = -\rho \kappa^{-1},
    \quad \mbox{for every } \rho \in (0, \infty).
  \end{equation*}
  In view of \cite[Theorem 1.2.7]{chen_random_2010}, it comes down to proving
  \begin{align}
    \label{eq:liminf_of_chi_n}
    \liminf_{n \to \infty} \frac{1}{n} \log\sum_{m=0}^{\infty} \frac{\theta^m}{m!}
    \left[ \expect\Big[ \Big\{ \sum_{k=1}^{n-1} \chi^n([0, k] \times [k, k+1]) \Big\}^m \Big] \right]^{\frac{1}{2}}
    \geq \frac{1}{2} \theta^2 \kappa, \\
    \label{eq:limsup_of_chi_n}
    \limsup_{n \to \infty} \frac{1}{n} \log\sum_{m=0}^{\infty} \frac{\theta^m}{m!}
    \left[ \expect\Big[ \Big\{ \sum_{k=1}^{n-1} \chi^n([0, k] \times [k, k+1]) \Big\}^m \Big] \right]^{\frac{1}{2}}
    \leq \frac{1}{2} \theta^2 \kappa,
  \end{align}

  We claim for every $n, m \in \N$
  \begin{equation}\label{eq:ineq_chi_greater_than_beta}
    \expect\Big[ \Big\{ \sum_{k=1}^{n-1} \chi^n([0, k] \times [k, k+1]) \Big\}^m \Big]
    \geq \expect\Big[ \Big\{ \sum_{k=1}^{n-1} \beta([0, k] \times [k, k+1]) \Big\}^m \Big],
  \end{equation}
  which proves \eqref{eq:liminf_of_chi_n} since we have
  \begin{equation*}
    \lim_{n \to \infty} \frac{1}{n} \log\sum_{m=0}^{\infty} \frac{\theta^m}{m!}
    \left[ \expect\Big[ \Big\{ \sum_{k=1}^{n-1} \beta([0, k] \times [k, k+1]) \Big\}^m \Big] \right]^{\frac{1}{2}}
    = \frac{1}{2} \theta^2 \kappa,
  \end{equation*}
  as shown in the proof of \cite[Theorem 4.3.1]{chen_random_2010}.
  To prove \eqref{eq:ineq_chi_greater_than_beta}, we set $A_n \defby \cup_{k=1}^{n-1} [0,k] \times[k, k+1]$ and observe
  \begin{align*}
    \MoveEqLeft[3]
    \expect\Big[ \Big\{ \sum_{k=1}^{n-1} \chi^n([0, k] \times [k, k+1]) \Big\}^m \Big] \\
    &= \lim_{\epsilon \to 0+} \expect[ \chi^n_{\epsilon}(A_n)^m] \\
    &= \lim_{\epsilon \to 0+} \expect\Big[ \Big( \int_{\R^2 \times A_n} e^{-\epsilon \abs{x}^2}
    e^{i\inp{x}{X^n_s - X^n_r}} dx dr ds \Big)^m \Big] \\
    &= \lim_{\epsilon \to 0+} \int_{(\R^2 \times A_n)^m} e^{-\epsilon \sum_{j=1}^m \abs{x_j}^2}
    \expect\Big[e^{i \sum_{j=1}^m \inp{x_j}{X^n_{s_j} - X^n_{r_j}}}\Big] \prod_{j=1}^m dx_j dr_j ds_j \\
    &\geq \lim_{\epsilon \to 0+} \int_{(\R^2 \times A_n)^m} e^{-\epsilon \sum_{j=1}^m \abs{x_j}^2}
    \expect\Big[e^{i \sum_{j=1}^m \inp{x_j}{B_{s_j} - B_{r_j}}}\Big] \prod_{j=1}^m dx_j dr_j ds_j \\
    &= \expect\Big[ \Big\{ \sum_{k=1}^{n-1} \beta([0, k] \times [k, k+1]) \Big\}^m \Big],
  \end{align*}
  where in the inequality we applied the estimate
  \begin{align*}
    \expect\Big[e^{i \sum_{j=1}^m \inp{x_j}{X^n_{s_j} - X^n_{r_j}}}\Big]
    &= \exp\Big( -\frac{1}{2} \expect\Big[ \Big(\sum_{j=1}^m \inp{x_j}{X^n_{s_j} - X^n_{r_j}}\Big)^2\Big] \Big) \\
    &= \exp\Big( -\frac{1}{2} \expect\Big[ \Big(\sum_{j=1}^m \inp{x_j}{B_{s_j} -B_{r_j}}\Big)^2\Big]
    + \frac{1}{2n} \Big\lvert \sum_{j=1}^m (s_j - r_j) x_j \Big\rvert^2 \Big) \\
    &\geq \exp\Big( -\frac{1}{2} \expect\Big[ \Big(\sum_{j=1}^m \inp{x_j}{B_{s_j} -B_{r_j}}\Big)^2\Big] \Big)
    = \expect\Big[e^{i \sum_{j=1}^m \inp{x_j}{B_{s_j} - B_{r_j}}}\Big].
  \end{align*}

  We move to the proof of \eqref{eq:limsup_of_chi_n}. To this end, we will prove
  \begin{equation}\label{eq:limsup_of_zeta}
    \limsup_{t \to \infty}\frac{1}{t} \log \P(\zeta \geq t) \leq - \kappa^{-1}.
  \end{equation}
  Indeed, if \eqref{eq:limsup_of_zeta} holds, then by the calculation in the first part we obtain
  \begin{equation*}
    \limsup_{n \to \infty} \frac{1}{n}
    \log \P \Big( \sum_{k=1}^{n-1} \chi^n([0, k] \times [k, k+1]) \geq \rho n^2 \Big) \leq -\rho \kappa^{-1}, \quad \rho \in (0, \infty),
  \end{equation*}
  and hence \eqref{eq:limsup_of_chi_n} follows from \cite[Theorem 1.2.9]{chen_random_2010}.
  However, by Chebyshev's inequality we have $\P(\zeta \geq t) \leq e^{-\lambda t} \expect[e^{\lambda \zeta}]$
  and hence $\limsup_{t \to \infty} \frac{1}{t} \log \P(\zeta \geq t) \leq -\lambda$ for every $\lambda \in (0, \kappa^{-1})$,
  which proves \eqref{eq:limsup_of_zeta}.
\end{proof}

\section{Construction of the IDS}\label{sec:ids}
The goal of this section is to provide the proof of Theorem \ref{thm:main}.
We begin with the construction of the IDS.
\begin{theorem}\label{thm:convergence_to_IDS}
  $\wnP$-almost surely, the family of measures $\{N_L(d\lambda)\}_{L \in \mathbb{I}}$, defined by \eqref{eq:def_of_N_L},
  converges vaguely to some limit $N(d\lambda)$ as $L \to \infty$.
  The limit measure is deterministic and characterized by the Laplace transform
  \begin{equation}\label{eq:laplace_transform_of_N}
    \int_{\R} e^{-t \lambda} N(d\lambda)
    = \frac{t^{\frac{t}{2\pi} - 1}}{2\pi} \expect[ e^{t \zeta} ],
    \quad t \in (0, \kappa^{-1}),
  \end{equation}
  where $\zeta = \zeta^1([0,1]_{\leq}^2)$ is the SILT of the Brownian bridge $X^1$ constructed in Section \ref{sec:brownian_bridge} and
  $\kappa$ is the best constant of Ladyzhenskaya's inequality \eqref{eq:GN_ineq}.
  In particular,
  \begin{equation}\label{eq:divergence_of_exponential_moment_of_N}
    \int_{\R} e^{-t \lambda} N(d\lambda) = \infty \qquad \mbox{for } t > \kappa^{-1}.
  \end{equation}
\end{theorem}
For the proof, we adopt the path integral approach as in \cite{nakao_spectral_1977} or \cite[Section VI.1.2]{carmona1990spectral}.
The advantage of this approach is to obtain a representation \eqref{eq:laplace_transform_of_N} of the Laplace transform.
A crucial difficulty in our setting is lack of useful Feynman-Kac formula.
Before the proof, we prepare two lemmas.
Recall the definition of $u$ from Definition \ref{def:def_of_u} and that
we have $u = \tilde{u}_{\infty}^0$ by Proposition \ref{prop:pam_in_plane}.
\begin{lemma}\label{lem:second_moment_of_pam}
  For $t \in (0, (2\kappa)^{-1})$, we have $\wnexpect[u(t, 0)^2] < \infty$ and
  \begin{equation}\label{eq:u_and_SILT}
    \wnexpect[u(t, 0)] = \frac{t^{\frac{t}{2\pi} - 1}}{2\pi} \expect[ e^{t \zeta} ]
  \end{equation}.
\end{lemma}
\begin{remark}
  The identity \eqref{eq:u_and_SILT} holds in $[0, \infty]$ for $t \in \R
  \setminus \{\kappa^{-1}\}$.
  See Corollary \ref{cor:moment_of_pam}.
\end{remark}
\begin{proof}
  Let $X^{t, j}$ ($j=1, 2$) be i.i.d. copies of $X^t$. By the Feynman-Kac formula,
  we have
  \begin{equation*}
    \wnexpect[ \tilde{u}_{\infty, \epsilon}^{\delta_0}(t, 0)^2]
    = \frac{1}{2 \pi t}
    \wnexpect\Big[ \expect\Big[ \exp \Big( \sum_{j=1}^2 \int_0^t \{ \xi_{\epsilon}(X_s^{t,j}) - c_{\epsilon} \} ds \Big)
    \Big] \Big].
  \end{equation*}
  We interchange the order or $\wnexpect$ and $\expect$ and we compute
  \begin{multline*}
    \wnexpect\Big[  \exp \Big( \sum_{j=1}^2 \int_0^t \{ \xi_{\epsilon}(X_s^{t,j}) - c_{\epsilon} \} ds \Big)
    \Big]  \\
    = \Big\{\prod_{j=1}^2 \exp\Big(\iint_{[0,t]_{\leq}^2} p_{\epsilon}(X^{t,j}_s - X^{t,j}_r) dr ds - c_{\epsilon} t \Big) \Big\} \times
    \exp\Big( \int_0^t \int_0^t p_{\epsilon}(X^{t, 1}_s - X^{t, 2}_r) dr ds \Big).
  \end{multline*}
  By Hölder's inequality and as $X^{t, 1}$ and $X^{t, 2}$ are independent,
  \begin{multline*}
    \expect\Big[
    \Big\{\prod_{j=1}^2 \exp\Big(\iint_{[0,t]_{\leq}^2} p_{\epsilon}(X^{t,j}_s - X^{t,j}_r) dr ds - c_{\epsilon} t \Big) \Big\} \times
        \exp\Big( \int_0^t \int_0^t p_{\epsilon}(X^{t, 1}_s - X^{t, 2}_r) dr ds \Big)
    \Big] \\
    \leq
    \expect\Big[
      \exp\Big(2 \iint_{[0,t]_{\leq}^2} p_{\epsilon}(X^{t}_s - X^{t}_r) dr ds - c_{\epsilon} t \Big) \Big\}
    \Big]
    \expect\Big[
    \exp\Big( 2 \int_0^t \int_0^t p_{\epsilon}(X^{t, 1}_s - X^{t, 2}_r) dr ds \Big)
    \Big]^{\frac{1}{2}}.
  \end{multline*}
  Recall the notation $\chi_{\epsilon}^t$ from
  Definition \ref{def:def_of_chi_epsilon}.
  By the scaling property of $X^t$, we have
  \begin{multline*}
    \expect\Big[
      \exp\Big(2 \iint_{[0,t]_{\leq}^2} p_{\epsilon}(X^{t}_s - X^{t}_r) dr ds - c_{\epsilon} t \Big) \Big\}
    \Big] \\
    = e^{2t(\expect[\chi^1_{t^{-1} \epsilon}([0,1]_{\leq}^2)] - c_{\epsilon})}
    \expect\big[ \exp\big\{2t (\chi_{t^{-1}\epsilon}^1([0,1]_{\leq}^2) - \expect[\chi_{t^{-1}\epsilon}^1([0,1]_{\leq}^2)]) \big\} \big].
  \end{multline*}
  By Lemma \ref{lem:renormalization_constant_of_SILT_of_BB}, we have
  $\expect[\chi^1_{t^{-1}\epsilon}([0,1]_{\leq}^2)] - c_{\epsilon} = \frac{1}{2\pi} \log t + o(1)$ as
  $\epsilon \to 0+$.
  By Theorem \ref{thm:construction_of_SILT_of_BB}-(iii),
  \begin{equation*}
    \sup_{\epsilon \in (0, \infty)}
    \expect\big[ \exp\big\{2t (\chi_{t^{-1}\epsilon}^1([0,1]_{\leq}^2) - \expect[\chi_{t^{-1}\epsilon}^1([0,1]_{\leq}^2)]) \big\} \big] < \infty
  \end{equation*}
  for $t \in (0, (2 \kappa)^{-1})$.
  Similarly, by applying Remark \ref{rem:uniform_integrability_of_e_to_alpha_bridge}, we obtain
  \begin{equation*}
    \sup_{\epsilon \in (0, \infty)}
    \expect\Big[
    \exp\Big( 2 \int_0^t \int_0^t p_{\epsilon}(X^{t, 1}_s - X^{t, 2}_r) dr ds \Big)
    \Big] < \infty
  \end{equation*}
  for $t \in (0, (2 \kappa)^{-1})$.
  Therefore,
  if $t \in (0, (2 \kappa)^{-1})$,
  \begin{equation*}
    \sup_{\epsilon \in (0, \infty)} \wnexpect[ \tilde{u}_{\infty, \epsilon}^{\delta_0}(t, 0)^2] < \infty
  \end{equation*}
  and the first claim follows by Fatou's lemma.

  We move to prove the second claim. The uniform integrability yields
  \begin{equation*}
    \wnexpect[u(t,0)]
    = \lim_{\epsilon \to 0+}
    \frac{1}{2\pi t} \wnexpect \times \expect \Big[ \exp\Big( \int_0^t (\xi_{\epsilon}(X_s^t) - c_{\epsilon}) ds \Big)
    \Big].
  \end{equation*}
  Since
  \begin{equation*}
    \wnexpect \Big[ \exp\Big(\int_0^t \xi_{\epsilon}(X^t_s) ds \Big) \Big]
    = e^{\chi^t_{\epsilon}([0,t]_{\leq}^2)},
  \end{equation*}
  we have
  \begin{equation*}
    \begin{aligned}
      \MoveEqLeft[3]
      \wnexpect[u(t,0)] \\
      &=\lim_{\epsilon \to 0+} \frac{1}{2\pi t} \expect[ e^{\chi^t_{\epsilon}([0,t]_{\leq}^2) - c_{\epsilon}t} ]\\
      &=\lim_{\epsilon \to 0+} \frac{1}{2 \pi t} e^{t(\expect[\chi^1_{t^{-1}\epsilon}([0,1]_{\leq}^2)] - c_{\epsilon})}
      \expect\big[ \exp\big\{t (\chi_{t^{-1}\epsilon}^1([0,1]_{\leq}^2) - \expect[\chi_{t^{-1}\epsilon}^1([0,1]_{\leq}^2)]) \big\} \big],
    \end{aligned}
  \end{equation*}
  where we applied the scaling property of $X^t$.
  Thus, the second claim follows from Lemma \ref{lem:renormalization_constant_of_SILT_of_BB} and
  Theorem \ref{thm:construction_of_SILT_of_BB}-(ii).
\end{proof}
\begin{lemma}\label{lem:laplace_trasform_converges}
  If $t \in (0, (2\kappa)^{-1})$, then $\wnP$-almost surely,
  \begin{equation*}
    \lim_{L \to \infty} \int_{\R} e^{-t \lambda} N_L(d\lambda) = \wnexpect[ u(t, 0) ].
  \end{equation*}
\end{lemma}
\begin{proof}
  We have by \eqref{eq:trace_of_exp_of_AH}
  \begin{align*}
    \int_{\R} e^{-t \lambda} N_L(d\lambda) &= \frac{1}{L^2} \sum_{n=1}^{\infty} e^{-t \lambda_{n, L}}
    = \frac{1}{L^2} \int_{Q_L} u_{L}^{x, 0}(t, x) dx \\
    &= \frac{1}{L^2} \int_{Q_L} u_{\infty}^{x, 0}(t, x) dx - \frac{1}{L^2} \int_{Q_L} \{
    u_{\infty}^{x, 0}(t, x) - u_L^{x, 0}(t, x) \} dx.
  \end{align*}
  By Lemma \ref{lem:shift_of_u_infty}, we have $u_{\infty}^{x, 0}(t, x) = u_{\infty}^{0, x}(t, 0)$.
  Therefore, Proposition \ref{prop:ergodicitiy_of_white_noise} and the ergodic theorem \cite[Theorem 10.14]{Kallenberg2002} yield
  \begin{equation*}
    \lim_{L \to \infty} \frac{1}{L^2} \int_{Q_L} u_{\infty}^{0, x}(t, 0) dx
     = \wnexpect[u(t, 0)] \quad \wnP\mbox{-a.s.}
  \end{equation*}
  Therefore, it suffices to show
  \begin{equation}\label{eq:convergence_of_u_L_to_u_infty}
    \lim_{L \to \infty} \frac{1}{L^2} \int_{Q_L} \{
    u_{\infty}^{x, 0}(t, x) - u_L^{x, 0}(t, x) \} dx = 0 \quad \wnP\mbox{-a.s.}
  \end{equation}
  Take $\delta \in (0, 1) \cap \Q$. We have
  \begin{align*}
    \int_{Q_L \setminus Q_{(1-\delta)L}} \{
    u_{\infty}^{x, 0}(t, x) - u_L^{x, 0}(t, x) \} dx
    &\leq \int_{Q_L \setminus Q_{(1-\delta)L}} u_{\infty}^{0, x}(t, 0) dx \\
    &\leq \sqrt{2\delta} L \Big( \int_{Q_L} u_{\infty}^{0, x}(t, 0)^2 dx \Big)^{\frac{1}{2}}.
  \end{align*}
  Thus, by the ergodic theorem,
  \begin{equation*}
    \limsup_{L \to \infty} \frac{1}{L^2} \int_{Q_L \setminus Q_{(1-\delta)L}} \{
    u_{\infty}^{x, 0}(t, x) - u_L^{x, 0}(t, x) \} dx
    \leq \sqrt{2 \delta} \wnexpect[u(t,0)^2]^{\frac{1}{2}} \quad \wnP\mbox{-a.s.}
  \end{equation*}
  We move to estimate $\frac{1}{L^2} \int_{Q_{(1-\delta)L}} \{
  u_{\infty}^{x, 0}(t, x) - u_L^{x, 0}(t, x) \} dx$.
  If $x \in Q_{(1-\delta)L}$, we have $-x + Q_L \supseteq Q_{\delta L}$ and hence the Feynman-Kac formula implies
  $u_L^{x, 0}(t, x) \geq u_{\delta L}^{0, x}(t, 0)$.
  If $L_0 \leq L$, then
  \begin{equation*}
    \frac{1}{L^2} \int_{Q_{(1-\delta)L}} \{
    u_{\infty}^{x, 0}(t, x) - u_L^{x, 0}(t, x) \} dx
    \leq \frac{1}{L^2} \int_{Q_{(1-\delta)L}} \{
    u_{\infty}^{0, x}(t, 0) - u_{\delta L_0}^{0, x}(t, 0) \} dx.
  \end{equation*}
  By the ergodic theorem again, we obtain, $\wnP$-a.s.,
  \begin{equation*}
    \limsup_{L \to \infty} \frac{1}{L^2} \int_{Q_{(1-\delta)L}} \{
    u_{\infty}^{x, 0}(t, x) - u_L^{x, 0}(t, x) \} dx
    \leq \wnexpect[u_{\infty}^{0,0}(t,0) - u_{\delta L_0}^{0,0}(t, 0)].
  \end{equation*}
  By the monotone convergence theorem, $\lim_{L_0 \to \infty} \wnexpect[u_{\delta L_0}^{0,0}(t, 0)]
  = \wnexpect[u_{\infty}^{0,0}(t, 0)]$.
  Therefore, by letting $L_0 \to \infty$ and then $\delta \to 0+$, we complete the proof of \eqref{eq:convergence_of_u_L_to_u_infty}
  and hence the proof of Lemma \ref{lem:laplace_trasform_converges}.
\end{proof}
\begin{proof}[Proof of Theorem \ref{thm:convergence_to_IDS}]
  By Lemma \ref{lem:laplace_trasform_converges}, if we set $T \defby (2\kappa)^{-1}$, there exists
  a $\wnP$-full subset $\S'_1$ of $\S'_0$
  such that for every $t \in (0, T) \cap \Q$ we have
  \begin{equation*}
    \lim_{L \to \infty} \int_{\R} e^{-t\lambda} N_L(d\lambda) = \wnexpect[ u(t, 0) ] \quad \mbox{ on } \S'_1.
  \end{equation*}
  In the rest of the proof, we restrict the probability space for $\wnP$ to $\S'_1$.
  We fix $\tau \in (0, T) \cap \Q$ and $\delta \in (0, \min\{\tau, T- \tau\})$ and define the probability measure $\nu_L$ by
  \begin{equation*}
    \nu_L(d\lambda) \defby \Big(\int_{\R} e^{-\tau \mu} N_L(d\mu) \Big)^{-1} e^{-\tau \lambda} N_L(d\lambda).
  \end{equation*}
  By Lemma \ref{lem:laplace_transform_of_probability_measures}, $\nu_L$ converges weakly to some limit $\nu$ and if we define $N$ by
  \begin{equation*}
    N(d\lambda) \defby \Big\{\lim_{L \to \infty} \int_{\R} e^{-\tau \mu} N_L(d\mu)\Big\} e^{\tau \lambda} \nu(d\lambda),
  \end{equation*}
  then $N_L$ converges vaguely to $N$ and
  \begin{equation*}
    \int_{\R} e^{-t \lambda} N(d\lambda) =
    \frac{t^{\frac{t}{2\pi} - 1}}{2\pi} \expect[e^{t \zeta}]
  \end{equation*}
  for $t \in (\tau - \delta, \tau + \delta)$.
  But by Lemma \ref{lem:analytic_extension_of_Laplace_transform}, this identity continues to hold for $t \in (0, \kappa^{-1})$.
  To prove the last claim \eqref{eq:divergence_of_exponential_moment_of_N}, suppose $\int_{\R} e^{-t\lambda} N(d\lambda) < \infty$
  for some $t > \kappa^{-1}$. Then, Lemma \ref{lem:analytic_extension_of_Laplace_transform} implies
  $\expect[e^{t \zeta}] < \infty$, which contradicts Theorem \ref{thm:right_tail_of_SILT_of_BB}.
\end{proof}

We next study the Lifshitz tails of the IDS $N$ constructed in Theorem \ref{thm:convergence_to_IDS}.
The goal is to prove the following theorem:
\begin{theorem}\label{thm:lifshitz_tails_white_noise}
  Let $N$ be the IDS of the Anderson Hamiltonian with two-dimensional white noise as constructed in Theorem \ref{thm:convergence_to_IDS}.
  Then, the equation \eqref{eq:lifshitz_tails_white_noise} holds.
\end{theorem}

From Theorem \ref{thm:convergence_to_IDS}, we easily deduce
\begin{equation}\label{eq:limsup_lifshitz_tails_white_noise}
  A \defby \limsup_{\lambda \to -\infty} \frac{1}{-\lambda} \log N(\lambda) = - \kappa^{-1}.
\end{equation}
In fact, $A \leq -\kappa^{-1}$ follows from Chebyshev's inequality. To show the opposite inequality, suppose $A < -\kappa^{-1}$. Then,
there exist $\epsilon \in (0, 1)$ and $C \in (0, \infty)$ such that
$N(\lambda) \leq C e^{(1+ \epsilon){\kappa}^{-1} \lambda}$ for $\lambda \in (-\infty, 0)$.
This implies $\int_{\R} e^{-t\lambda} N(d\lambda) < \infty$ for $t \in (\kappa^{-1}, (1+ \epsilon){\kappa}^{-1})$,
which contradicts Theorem \ref{thm:convergence_to_IDS}.

To the best of the author's knowledge, the proof of the Lifshitz tails of the form
\begin{equation*}
  \lim_{\lambda \to \lambda_0} (-\lambda)^{a} \log N^V(\lambda) = b,
\end{equation*}
where $\lambda_0$ is the bottom of the spectrum,
usually relies on the asymptotics of the Laplace transform
\begin{equation*}
  \lim_{t \to \infty} t^{\tilde{a}} \log \int_{\R} e^{-t\lambda} N(d\lambda) = \tilde{b}.
\end{equation*}
Indeed, the approach to study the asymptotics of the Laplace transform was first implemented in the case where
the random potential $V$ is of the form
\begin{equation*}
  V(x) = \int_{\R^d} \phi(x-y) \mu(dy),
\end{equation*}
where $\mu$ is the Poisson random measure with respect to the Lebesgue measure and $\phi$ is a nonnegative continuous function
which is not identically zero and satisfies $\phi(x) = o(\abs{x}^{-d-2})$ as $x \to \infty$.
The work \cite{donsker_asymptotics_1975} (see \cite{nakao_spectral_1977} as well) studied the asymptotics
of the Laplace transform of the IDS $N^V$ and consequently proved
\begin{equation*}
  \lim_{\lambda \to 0+} \lambda^{\frac{d}{2}} \log N^V(\lambda) = - \gamma_d^{\frac{d}{2}},
\end{equation*}
where $\gamma_d$ is the smallest eigenvalue of $-\frac{1}{2} \Delta$ on the unit ball in $\R^d$ with Dirichlet boundary conditions.
This approach works for stationary, ergodic centered Gaussian potentials $V$ which are not too singular.
See Appendix \ref{sec:lifshiz_tail_riesz} for more detail.

However, this approach does not work in the case of the two-dimensional white noise, since the Laplace transform blows up.
We remark that the work \cite{benaim_smile_2008} studied tail behaviors of distribution functions
which have finite exponential moment only for small parameter $t > 0$.
In order to apply the result from \cite{benaim_smile_2008} to our problem, we need to understand the behavior of
\begin{equation*}
  \int_{\R} e^{-t\lambda} N(d\lambda) = \frac{t^{\frac{t}{2\pi} - 1}}{2\pi} \expect[e^{t\zeta}]
\end{equation*}
as $t \to \kappa^{-1}-$. This seems very difficult, and therefore we take a different approach.

In view of \eqref{eq:limsup_lifshitz_tails_white_noise}, the proof of Theorem \ref{thm:lifshitz_tails_white_noise} comes down to
showing \eqref{eq:liminf_lifshitz_tails_white_noise}.
To obtain a lower bound of $N(\lambda)$, we use the inequality $N(\lambda) \geq \wnexpect[N_1(\lambda)]$,
which will be proved in Lemma \ref{lem:Dirichlet_comparisons}. This inequality is obtained by a well-known technique to use superadditivity of the Dirichlet eigenvalue counting functions
(\cite{kirsch_large_1983}, \cite{simon_lifschitz_1985}, \cite[Section 3]{ids2007}).
Then, we prove \eqref{eq:liminf_lifshitz_tails_white_noise}
by using the tail estimate of
the first eigenvalue of the Anderson Hamiltonian obtained in \cite{chouk2020asymptotics}.
\begin{lemma}\label{lem:Dirichlet_comparisons}
  For every $\lambda \in \R$ and $L \in (0, \infty)$, we have
  $N(\lambda) \geq \wnexpect[ N_L(\lambda)]$.
\end{lemma}
\begin{proof}
  We set
  \begin{equation*}
  N_{L, \epsilon}(\lambda) \defby L^{-2} \sum_{n=1}^{\infty} \indic_{\{
  \lambda_{n, L}^{\bm{\xi}_{L, \epsilon}^{\n}} \leq \lambda \}}
  \end{equation*}
  and denote by $N_{\infty, \epsilon}$ the IDS of the
  Anderson Hamiltonian $\H_{\epsilon} = -\frac{1}{2} \Delta - (\xi_{\epsilon} - c_{\epsilon})$.
  By the continuity of the eigenvalues with respect to the enhanced noise
  \cite[Theorem 5.4]{chouk2020asymptotics},
  we have $\lim_{\epsilon \to 0+}\lambda_{n, L}^{\bm{\xi}_{L, \epsilon}^{\n}} = \lambda_{n, L}^{\bm{\xi}_{L}^{\n}}$ for every $n$ and hence
  $N_{L, \epsilon}(d\lambda)$ converges to $N_L(d\lambda)$ vaguely.
  Since
  \begin{equation*}
    \int_{\R} e^{-t\lambda} N_{\infty, \epsilon} (d\lambda) = \frac{1}{2\pi t}
    \wnexpect[ e^{\chi^t_{\epsilon}([0,t]_{\leq}^2) - c_{\epsilon}t} ],
  \end{equation*}
  Theorem \ref{thm:construction_of_SILT_of_BB}-(ii) implies
  $\lim_{\epsilon \to 0+} \int_{\R} e^{-t\lambda} N_{\infty, \epsilon}(d\lambda) = \int_{\R} e^{-t\lambda} N(d\lambda)$
  for $t \in (0, \kappa^{-1})$.
  Thus, by applying the same trick used in the proof of Theorem \ref{thm:convergence_to_IDS}
  to prove the vague convergence of $N_L$, we conclude $N_{\infty, \epsilon}(d\lambda)$ converges vaguely to $N(d\lambda)$.
  We remark that $N_{\infty, \epsilon}(d \lambda)\vert_{(-\infty, 0)}$ converges to
  $N_{\infty, \epsilon}(d \lambda)\vert_{(-\infty, 0)}$ weakly.

  Let $\mu > \lambda$. For each $\epsilon \in (0, \infty)$, we have
  $N_{\infty, \epsilon}(\mu) \geq \wnexpect[N_{L, \epsilon}(\mu)]$. See \cite[Section VI.1.3]{carmona1990spectral}, especially (VI.17).
  Therefore, by letting $\epsilon \to 0+$ and applying Fatou's lemma,
we obtain $N(\mu) \geq \wnexpect[N_L(\lambda)]$. By letting \change{$\mu \to \lambda+$},
we obtain $N(\lambda) \geq \wnexpect[N_L(\lambda)]$.
\end{proof}
\begin{lemma}\label{lem:eigenvalue_tail_estimate}
  Let $L \in (0, \infty)$ and $\mu > \kappa^{-1}$. Then, there exist $c \in (0, \infty)$ and
   $\Lambda \in (-\infty, 0)$ such that
  \begin{equation*}
    \wnP(\lambda_{1, L} \leq \lambda) \geq 1 - e^{ c \lambda e^{\mu \lambda}}, \quad \mbox{for } \lambda \leq \Lambda.
  \end{equation*}
\end{lemma}
\change{
\begin{proof}
  We denote by $\bm{\lambda}(Q_L, \beta)$ the largest eigenvalue of $\Delta + \beta \xi$. Then, one has
  \begin{equation}\label{eq:lambda_bm_lambda}
    \lambda_{1, L} = - \frac{1}{2} \bm{\lambda}(Q_L, 2) \quad  \text{in law}.
  \end{equation}
  By \cite[Theorem~2.6 and Theorem~2.7]{chouk2020asymptotics},
  if $\mu' > 2 \kappa^{-1}$, there exist $r \in (0, \infty)$ and $M \in (0, \infty)$ such that
  \begin{equation}\label{eq:bm_lambda_tail}
    \P(\bm{\lambda}(Q_L, 2) \geq 2 x)
    \geq 1 - \exp \Big(- \frac{2 x e^{2 \log L - \frac{\mu'}{2^2}(2 x)}}{2 r^2} \Big)
    = 1 - \exp \Big(- \frac{ x e^{2 \log L - \frac{\mu'}{2} x}}{ r^2} \Big)
  \end{equation}
  for every $x \geq M$. Therefore, \eqref{eq:lambda_bm_lambda} and \eqref{eq:bm_lambda_tail} prove the claim
  ($\mu = \frac{\mu'}{2}$, $c= \frac{e^{2 \log L}}{r^2}$).
\end{proof}
}

\begin{proof}[Proof of Theorem \ref{thm:lifshitz_tails_white_noise}]
  By Lemma \ref{lem:Dirichlet_comparisons}, we have
  $N(\lambda) \geq \wnexpect[N_1(\lambda)] \geq \wnP(\lambda_{1, 1} \leq \lambda)$.
  Let $\mu > \kappa^{-1}$. If $-\lambda$ is sufficiently large, then Lemma~\ref{lem:eigenvalue_tail_estimate} yields
  \begin{equation*}
    N(\lambda) \geq 1 - e^{c \lambda e^{\mu \lambda}} \geq \frac{c (-\lambda)}{2} e^{\mu \lambda}.
  \end{equation*}
  Therefore, $\liminf_{\lambda \to -\infty} (-\lambda)^{-1} \log N(\lambda) \geq - \mu$.
  Since $\mu > \kappa^{-1}$ is arbitrary, we complete the proof of \eqref{eq:liminf_lifshitz_tails_white_noise}
  and hence the proof of Theorem \ref{thm:lifshitz_tails_white_noise}.
\end{proof}

Finally, we apply Theorem \ref{thm:convergence_to_IDS} to the moment of the PAM in the plane.
\begin{corollary}\label{cor:moment_of_pam}
  As before, let $u$ be the solution of the PAM in $\R^2$ with initial condition $\delta_0$.
  Then, we have
  \begin{equation*}
    \wnexpect[u(t, 0)]
    \begin{cases}
      < \infty \qquad &\mbox{if } t < \kappa^{-1}, \\
      = \infty &\mbox{if } t > \kappa^{-1}.
    \end{cases}
  \end{equation*}
  Furthermore, for $t \in \R \setminus \{\kappa^{-1}\}$, we have
  \begin{equation*}
    \wnexpect[u(t, 0)] = \int_{\R} e^{-\lambda t} N(d \lambda)
    = \frac{t^{\frac{t}{2 \pi} - 1}}{2 \pi}
    \expect[e^{t \zeta}].
  \end{equation*}
\end{corollary}
\begin{proof}
  If $t < \kappa^{-1}$, by Fatou's lemma, the Feynman-Kac formula and Theorem \ref{thm:construction_of_SILT_of_BB}-(iii), we have
  \begin{equation}\label{eq:upper_bound_u_by_SILT}
    \wnexpect[u(t, 0)] \leq \liminf_{\epsilon \to 0+} \wnexpect[ u_{\epsilon}(t, 0)]
    = \frac{t^{\frac{t}{2\pi} - 1}}{2 \pi} \expect[e^{t \zeta}] < \infty.
  \end{equation}
  Suppose $t > 0$.
  If $\phi \in C_c^{\infty}(\R)$ with $\norm{\phi}_{\infty} \leq 1$,
  we have
  \begin{equation*}
    \int_{\R} e^{-t\lambda} \phi(\lambda) N(d\lambda)
    = \lim_{L \to \infty} \int_{\R} e^{-t \lambda} \phi(\lambda) N_L(d\lambda)
    \leq \liminf_{L \to \infty} \frac{1}{L^2} \int_{Q_L} u_{\infty}^{0, x}(t, 0) dx,
  \end{equation*}
  where we used the identity \eqref{eq:trace_of_exp_of_AH}.
  By the ergodic theorem, as no assumption of integrability is needed, we have
  \begin{equation*}
    \lim_{L \to \infty} \frac{1}{L^2} \int_{Q_L} u_{\infty}^{0, x}(t, 0) dx = \wnexpect[u_{\infty}^{0, 0}(t, 0)],
    \qquad \wnP\mbox{-a.s.}
  \end{equation*}
  Therefore, recalling Proposition \ref{prop:pam_in_plane}, we conclude
  \begin{equation}\label{eq:lower_bound_of_u_by_IDS}
    \int_{\R} e^{-t \lambda} N(d\lambda) \leq \wnexpect[u(t, 0)].
  \end{equation}
  Thus, the claims of Corollary \ref{cor:moment_of_pam} follows from
  \eqref{eq:upper_bound_u_by_SILT}, \eqref{eq:lower_bound_of_u_by_IDS} and
  Theorem \ref{thm:convergence_to_IDS}.

  Alternatively, we can prove the blow-up of the moment by applying a result on the first eigenvalue asymptotics.
  Indeed, we have
  \begin{equation*}
    \frac{1}{L^2} e^{-t\lambda_{1,L}} \leq \int_{\R} e^{-t\lambda} N_L(d\lambda) \leq \frac{1}{L^2} \int_{Q_L} u_{\infty}^{0, x}(t, 0) dx.
  \end{equation*}
  By Lemma \ref{lem:eigenvalue_tail_estimate}, the moment of the left hand side explodes if $t > \kappa^{-1}$.
\end{proof}
\begin{remark}\label{rem:moment_of_finite_volume_pam}
  As the proof above suggests, Corollary \ref{cor:moment_of_pam} holds if $u$ is replaced by
  $u_L^{0, 0}$.
\end{remark}
\begin{remark}\label{rem:moment_of_pam_with_general_initial_condition}
  For $t \in (0, \kappa^{-1})$, we have $\wnexpect[u(t, x)] \leq \frac{t^{\frac{t}{2\pi} -1}}{2\pi} \expect[e^{t \zeta}] < \infty$.
  Indeed, we have
  \begin{align*}
    \wnexpect[u(t, x)] &\leq \liminf_{\epsilon \to 0+} \wnexpect[\tilde{u}^{\delta_0}_{\infty, \epsilon}(t, x)] \\
    &= \liminf_{\epsilon \to 0+} p_t(x) \expect\Big[ \exp \Big( \iint_{[0, t]_{\leq}^2} p_{\epsilon}
    \Big(X^t_s - X^t_r - \frac{(s-r)x}{t} \Big) ds dr - c_{\epsilon}t \Big) \Big].
  \end{align*}
  As in the proof of \eqref{eq:ineq_chi_greater_than_beta}, we observe
  \begin{align*}
    \MoveEqLeft[3]
    \expect\Big[ \Big( \iint_{[0, t]_{\leq}^2} p_{\epsilon}
    \Big(X^t_s - X^t_r - \frac{(s-r)x}{t} \Big) ds dr \Big)^m \Big] \\
    &= \int_{(\R^2 \times [0, t]_{\leq}^2)^m} e^{-\epsilon \sum_{j=1}^m \abs{x_j}^2 }
    \expect[ e^{i \sum_{j=1}^m \inp{x_j}{X^t_{s_j} - X^t_{r_j} - \frac{s_j - r_j}{t} x}}] \prod_{j=1}^m ds_j dr_j dx_j \\
    &\leq \int_{(\R^2 \times [0, t]_{\leq}^2)^m} e^{-\epsilon \sum_{j=1}^m \abs{x_j}^2 }
    \expect[ e^{i \sum_{j=1}^m \inp{x_j}{X^t_{s_j} - X^t_{r_j}}}] \prod_{j=1}^m ds_j dr_j dx_j \\
    &= \expect\Big[ \Big( \iint_{[0, t]_{\leq}^2} p_{\epsilon}
    (X^t_s - X^t_r) ds dr \Big)^m \Big]
  \end{align*}
  and therefore
  \begin{equation*}
    \expect\Big[ \exp \Big( \int_{[0, t]_{\leq}^2} p_{\epsilon}
    \Big(X^t_s - X^t_r - \frac{(s-r)x}{t} \Big) ds dr - c_{\epsilon}t \Big) \Big]
    \leq \expect[e^{\chi_{\epsilon}^t([0,t]_{\leq}^2) - c_{\epsilon} t}].
  \end{equation*}
  Hence,
  \begin{equation*}
    \wnexpect[u(t, x)] \leq \liminf_{\epsilon \to 0+} \frac{1}{2\pi t} \expect[e^{\chi_{\epsilon}^t([0, t]^2_{\leq})
    - c_{\epsilon} t} ] = \frac{t^{\frac{t}{2\pi} - 1}}{2 \pi} \expect[e^{t \zeta}] < \infty.
  \end{equation*}
  However, we do not know whether $\wnexpect[u(t, x)] = \infty$ for $t > \kappa^{-1}$.

  Similarly, in view of \eqref{eq:uniform_integrability_of_e_to_beta}, we have
  $\wnexpect[\tilde{u}^{\phi}_{\infty}(t, x)] < \infty$ if $\phi \in L^{\infty}(\R^2)$ and $t \in (0, \kappa^{-1})$.
  It is not known whether the moment blows up for $t > \kappa^{-1}$ with general nonnegative initial condition,
  especially for $\phi = \indic$.
  Furthermore, we do not know whether the moment blows up or not at $t = \kappa^{-1}$.
\end{remark}
\begin{remark}\label{rem:extension_to_other_settings}
  The work \cite{labbe_continuous_2019} constructed the Anderson Hamiltonian
  with three-dimensional white noise on the box $[-\frac{L}{2}, \frac{L}{2}]^3$ and showed that the spectrum consists of
  the eigenvalues $\{\lambda_{n, L}^{\operatorname{3D}}\}_{n=1}^{\infty}$.
  It is natural to extend our work to this three-dimensional settings, namely
  to prove the convergence of
  $N_L^{\operatorname{3D}}(\lambda) \defby L^{-3} \sum_{n=1}^{\infty} \indic_{\{\lambda_{n,L}^{\operatorname{3D}} \leq \lambda\}}$
  and the Lifshitz tail asymptotics.
  We remark that the path integral approach is not available in this setting, as we expect
  $\int_{\R} e^{-t\lambda} N^{\operatorname{3D}}(d\lambda) = \infty$ for every $t \in \R$.
  In fact, we should expect from
  Remark \ref{rem:variations} the asymptotics
  \begin{equation}\label{eq:lifshitz_tails_3d_white_noise}
    \lim_{\lambda \to -\infty} (-\lambda)^{-\frac{1}{2}} \log N^{\operatorname{3D}}(\lambda)
    = - \frac{2 \sqrt{2} }{3\sqrt{3} \kappa(3,3)},
  \end{equation}
  where $\kappa(3, 3)$ is the best constant of the inequality
  $\norm{f}_{L^4(\R^3)}^4 \leq C \norm{f}_{L^2(\R^3)} \norm{\nabla f}_{L^2(\R^3)}^3$.
  For the construction of the IDS in this setting, we believe that
  the functional analytic approach \cite[Section VI.1.3]{carmona1990spectral} should be adopted.
  We also believe that the lower bound of \eqref{eq:lifshitz_tails_3d_white_noise} can be obtained as in the proof
  of \eqref{eq:liminf_lifshitz_tails_white_noise}, if we have a corresponding result of Lemma \ref{lem:eigenvalue_tail_estimate}
  for the three-dimensional white noise.

  \delete{
  In view of \cite{mouzard2020weyl}, it is interesting to extend our work to Anderson Hamiltonians on more general
  two-dimensional manifolds. A natural question is how the geometry influences the critical exponent of the corresponding IDS or
  equivalently the blow-up time of the moment of the PAM.
  }
\end{remark}

\appendix
\section{Lifshitz tails of Anderson Hamiltonians with generalized Gaussian potentials}\label{sec:lifshiz_tail_riesz}
Let $V$ be the generalized centered Gaussian field in $\R^d$ whose covariance is given by
\begin{equation*}
  \expect^V[ V(x) V(y) ] = \nu \abs{x-y}^{-\sigma} \qquad \sigma \in (0, \min\{2, d\}).
\end{equation*}
We denote by $V_{L}^{\n}$ the restriction of $V$ on $Q_L=[-\frac{L}{2}, \frac{L}{2}]^d$ with Dirichlet boundary conditions.
Since the regularity of $V$ is better than $-\frac{\sigma}{2} - \epsilon$ for every $\epsilon > 0$, $V_L^{\n}$ can be
enhanced to $\bm{V}_L^{\n}$ without renormalization, and hence we can define $\H_L^{\bm{V}_L^{\n}}$.
As in Section \ref{sec:ids}, we can prove that the eigenvalue counting measure associated with the distribution function
\begin{equation*}
  N_L^V(\lambda) \defby \frac{1}{L^2} \sum_{n=1}^{\infty} \indic_{\{\lambda_{n, L}^{\bm{V}_L^{\n}} \leq \lambda \}}
\end{equation*}
converges vaguely to $N^V(d\lambda)$ almost surely.
The measure $N^V(d\lambda)$ is characterized by the identity
\begin{equation}\label{eq:laplace_transform_of_N_V}
  \int_{\R} e^{-t\lambda} N^{V}(d\lambda) = \frac{1}{(2 \pi t)^{\frac{d}{2}}}
  \expect \Big[ \exp\Big(\frac{\nu}{2} \int_0^t \int_0^t \frac{dr ds}{\abs{X_r^t - X_s^t}^{\sigma}} \Big) \Big].
\end{equation}

The goal of this section is to prove the Lifshitz tail of $N^V$.
We define the distribution function by $N^V(\lambda) \defby N^V((-\infty, \lambda])$ and set
\begin{equation*}
  \rho(d, \sigma) \defby \sup_{\norm{f}_{L^2(\R^d)} = 1}
  \int_{\R^d} \Big( \int_{\R^d} \frac{f(\lambda + \gamma) f(\gamma)}{\sqrt{1 + 2^{-1} \abs{\lambda + \gamma}^2} \sqrt{1 +
  2^{-1} \abs{\gamma}^2}} d\gamma
  \Big)^2 \frac{C_{d, \sigma}}{\abs{\lambda}^{d- \sigma}} d\lambda,
\end{equation*}
where $C_{d, \sigma} \defby \pi^{-\frac{d}{2}} 2^{-\sigma} \Gamma(\frac{d-\sigma}{2}) \Gamma(\frac{\sigma}{2})^{-1}$.
We note that $(2\pi)^{d} C_{d, \sigma} \abs{\cdot}^{-d + \sigma}$ is the Fourier transform of $\abs{\cdot}^{-\sigma}$:
$\int_{\R^d} \abs{x}^{-\sigma} e^{-i \inp{x}{y}} dx = (2\pi)^{d} C_{d, \sigma} \abs{y}^{-d + \sigma}$.
\begin{theorem}\label{thm:lifshitz_tail_riesz}
  We have
  \begin{equation*}
    \lim_{\lambda \to -\infty} (-\lambda)^{-\frac{4-\sigma}{2}} \log N^V(\lambda) = - \frac{1}{2 \nu \rho(d, \sigma)}.
  \end{equation*}
\end{theorem}
\delete{
\begin{remark}\label{rem:lifshitz_tail_smooth_gaussian}
  Let $\tilde{V}$ be a centered, ergodic and stationary Gaussian field on $\R^d$ such that
  the covariance $\tilde{\gamma}(x) = \expect^{\tilde{V}}[\tilde{V}(0)\tilde{V}(x)]$ is continuous
  and let $N^{\tilde{V}}$ be the IDS of $-\frac{1}{2} \Delta - \tilde{V}$.
  Then, as shown in \cite[Proposition VI.2.2]{carmona1990spectral}, we have
  \begin{equation*}
    \lim_{\lambda \to -\infty} (-\lambda)^{-2} \log N^{\tilde{V}}(\lambda) = -\frac{1}{2\tilde{\gamma}(0)}.
  \end{equation*}
\end{remark}
}
\begin{remark}\label{rem:variations}
  We review relations among variations. As in \cite[Appendix A.3]{chen2014}, we set
  \begin{align}
    \kappa(d, \sigma) &\defby \inf \Big\{ C \in (0, \infty) \Big\vert
    \int_{\R^d \times \R^d} \frac{f(x)^2 f(y)^2}{\abs{x-y}^{\sigma}} dx dy \notag \\
    &\hspace{4cm} \leq C \norm{f}_{L^2(\R^d)}^{4-\sigma} \norm{\nabla f}_{L^2(\R^d)}^{\sigma},
    \,\, \forall f \in C_c^{\infty}(\R^d) \Big\}, \label{eq:def_kappa_d_sigma}\\
    M(d, \sigma) &\defby \sup_{f \in C_c^{\infty}(\R^d), \,\, \norm{f}_{L^2(\R^d)} = 1} \Big\{
    \Big(\int_{\R^d \times \R^d} \frac{f(x)^2 f(y)^2}{\abs{x-y}^{\sigma}} dx dy \Big)^{\frac{1}{2}}
    - \frac{1}{2} \int_{\R^d} \abs{\nabla f(x)}^2 dx \Big\} \notag.
  \end{align}
  By \cite[Lemma A.4]{chen2014}, we have
  \begin{equation*}
    M(d, \sigma) = \frac{4-\sigma}{4} \Big(\frac{\sigma}{2} \Big)^{\frac{\sigma}{4-\sigma}}
    \kappa(d, \sigma)^{\frac{2}{4-\sigma}}.
  \end{equation*}
  By \cite[Theorem 1.5]{bass2009}, we have
  \begin{equation*}
    \rho(d, \sigma) =  M(d, \sigma)^{2- \frac{\sigma}{2}}.
  \end{equation*}
  We remark that \cite[Theorem 1.5]{bass2009} wrongly claims ``$\rho(d, \sigma) = (2\pi)^{-d}
  M(d, \sigma)^{2- \frac{\sigma}{2}}$". Indeed, the expression
  $(2\pi)^{-d(p+1)}$ and $(2\pi)^{d(p-1)}$ in \cite[(7.22)]{bass2009}
  should be corrected to $(2\pi)^{-dp}$ and $(2\pi)^{dp}$ respectively.
  Therefore, we have
  \begin{equation}\label{eq:kappa_and_rho}
    \rho(d, \sigma) = \Big( \frac{4-\sigma}{4} \Big)^{\frac{4-\sigma}{2}} \Big( \frac{\sigma}{2} \Big)^{\frac{\sigma}{2}} \kappa(d, \sigma).
  \end{equation}

  \change{Let us have a heuristic discussion about how Theorem~\ref{thm:lifshitz_tail_riesz} is consistent with other results.
  We first consider the case of the two-dimensional white noise $\xi$.
  Since we have $\expect[\xi(x)\xi(y)] = \delta_{\R^2}(x-y)$ heuristically and
  since the two-dimensional Dirac's function $\delta_{\R^2}$ satisfies the scaling relation
  \begin{equation*}
    \delta_{\R^2}(\lambda \cdot) = \lambda^{-2} \delta_{\R^2}(\cdot), \quad \lambda>0,
  \end{equation*}
  we should set $\sigma = 2$.
  Furthermore, we should set
  \begin{equation*}
    \kappa(2, 2) \defby
    \inf \Big\{ C \in (0, \infty) \Big\vert
    \int_{\R^d} f(x)^4 d x
    \leq C \norm{f}_{L^2(\R^d)}^{2} \norm{\nabla f}_{L^2(\R^d)}^{2},
    \,\, \forall f \in C_c^{\infty}(\R^d) \Big\};
  \end{equation*}
  namely in \eqref{eq:def_kappa_d_sigma} we replaced $\frac{1}{\abs{x-y}^{\sigma}}$ by $\delta_{\R^2}(x - y)$.
  In this setting,  \eqref{eq:kappa_and_rho} becomes $\rho(2, 2) = 2^{-1} \kappa(2, 2)$ and
  this shows consistency between Theorem \ref{thm:lifshitz_tail_riesz} and \eqref{eq:lifshitz_tails_white_noise}.
   }
  Similarly, the case where $V$ is a continuous Gaussian field corresponds to $\sigma = 0$ and \eqref{eq:kappa_and_rho} becomes
  $\rho(d, 0) = \kappa(d, 0) = 1$, which is consistent with \change{the result of \cite{kirsch_ids_82} mentioned in Remark~\ref{rem:gaussian_lifshitz}}.
  The case where $V$ is the one-dimensional white noise corresponds to $(d, \sigma) = (1, 1)$ and,
  since $\kappa(1, 1)$ is computable \cite[Theorem C.4]{chen_random_2010},
  \eqref{eq:kappa_and_rho} becomes $\rho(1, 1) = \frac{3}{8\sqrt{2}}$,
  which is consistent with the result in \cite{fukushima_spectra_1977}.
  The case where $V$ is the three-dimensional white noise corresponds to $(d, \sigma) = (3, 3)$ and hence
  this justifies the conjecture \eqref{eq:lifshitz_tails_3d_white_noise}.
\end{remark}
Theorem \ref{thm:lifshitz_tail_riesz} essentially follows from Theorem \ref{thm:brownian_intersection} below,
combined with technical argument (Lemma \ref{lem:lifshitz_tail_lower_bound} and Lemma \ref{lem:lifshitz_tail_upper_bound})
to replace the Brownian motion by the Brownian bridges and a Tauberian theorem (Proposition \ref{prop:minlos_tauberian}).
\begin{theorem}[{\cite[Theorem 1.1]{chen_large_2010}}]\label{thm:brownian_intersection}
  We have
  \begin{equation*}
    \lim_{\theta \to \infty} \theta^{-\frac{2}{2-\sigma}} \log
    \expect\Big[ \exp\Big(\theta \int_0^1 \int_0^1 \frac{dsdr}{\abs{B_s - B_r}^{\sigma}}\Big) \Big]
    = 2^{\frac{6}{2-\sigma}} (2-\sigma) (4-\sigma)^{-\frac{4-\sigma}{2-\sigma}} \rho(d, \sigma)^{\frac{2}{2-\sigma}}.
  \end{equation*}
\end{theorem}
\begin{lemma}\label{lem:lifshitz_tail_lower_bound}
  For every $\theta \in (0, \infty)$, we have
  \begin{equation*}
    \expect\Big[ \exp\Big( \theta \int_0^1 \int_0^1 \frac{dsdr}{\abs{B_s-B_r}^{\sigma}}\Big) \Big]
    \leq \expect\Big[ \exp\Big( \theta \int_0^1 \int_0^1 \frac{dsdr}{\abs{X_s^1 - X_r^1}^{\sigma}}\Big) \Big]
  \end{equation*}
\end{lemma}
\begin{proof}
  It suffices to show that
  \begin{equation*}
    \expect\Big[ \Big( \int_0^1 \int_0^1 \frac{dsdr}{\abs{B_s-B_r}^{\sigma}} \Big)^k \Big]
    \leq \expect\Big[ \Big( \int_0^1 \int_0^1 \frac{dsdr}{\abs{X_s^1 - X_r^1}^{\sigma}} \Big)^k \Big]
  \end{equation*}
  for every $k \in \N$.
  Let $Y_t = B_t$ or $Y_t = X_t^1$. Then, we have
  \begin{equation*}
    \int_0^1 \int_0^1 \frac{dsdr}{\abs{Y_s - Y_r}^{\sigma}}
    = C_{d, \sigma} \int_0^1 \int_0^1 \int_{\R^d} \frac{1}{\abs{\xi}^{d-\sigma}} e^{i \inp{Y_s - Y_r}{\xi}} d\xi ds dr
  \end{equation*}
  and
  \begin{multline*}
    \expect\Big[ \Big( \int_0^1 \int_0^1 \frac{dsdr}{\abs{Y_s-Y_r}^{\sigma}} \Big)^k \Big]\\
    = C^k_{d, \sigma} \int_{([0,1]^2 \times \R^d)^k} \Big( \prod_{j=1}^k \abs{\xi_j}^{-d+\sigma} \Big)
    \expect\Big[ \exp\Big(i \sum_{j=1}^k \inp{Y_{s_j} - Y_{r_j}}{\xi_j} \Big)\Big] \prod_{j=1}^k ds_j dr_j d\xi_j.
  \end{multline*}
  We note
  \begin{equation*}
    \expect\Big[ \exp\Big(i \sum_{j=1}^k \inp{Y_{s_j} - Y_{r_j}}{\xi_j} \Big)\Big]
    = \exp\Big(-\frac{1}{2} \expect\Big[\Big(\sum_{j=1}^k \inp{Y_{s_j} - Y_{r_j}}{\xi_j}\Big)^2\Big] \Big).
  \end{equation*}
  Therefore, it remains to observe
  \begin{align*}
    \expect\Big[ \Big(\sum_{j=1}^k \inp{X_{s_j}^1 - X_{r_j}^1}{\xi_j}\Big)^2 \Big]
    &= \expect\Big[ \Big(\sum_{j=1}^k \inp{B_{s_j} - B_{r_j}}{\xi_j}\Big)^2 \Big] - \Big\vert \sum_{j=1}^k (s_j - r_j) \xi_j \Big\vert^2 \\
    &\leq \expect\Big[ \Big(\sum_{j=1}^k \inp{B_{s_j} - B_{r_j}}{\xi_j}\Big)^2 \Big]. \qedhere
  \end{align*}
\end{proof}
\begin{lemma}\label{lem:lifshitz_tail_upper_bound}
  We have
  \begin{equation*}
    \limsup_{\theta \to \infty} \theta^{-\frac{2}{2-\sigma}} \log \expect\Big[ \exp\Big(\theta \int_0^1 \int_0^1
    \frac{dsdr}{\abs{X_s^1 - X_r^1}^{\sigma}} \Big) \Big] \leq
    2^{\frac{6}{2-\sigma}} (2-\sigma) (4-\sigma)^{-\frac{4-\sigma}{2-\sigma}} \rho(d, \sigma)^{\frac{2}{2-\sigma}}.
  \end{equation*}
\end{lemma}
\begin{proof}
  We denote by $\tilde{\kappa}$ the right hand side of the claimed inequality.
  We fix large $n \in \N$ and set $\delta \defby \frac{1}{n}$.
  We take $p_1, p_2 \in (1, \infty)$ such that $p_1^{-1} + p_2^{-1} = 1$.
  Later we let $n \to \infty$ and then $p_1 \to 1 +$.
  By H\"older's inequality, we have
  \begin{equation*}
    \expect\Big[ \exp \Big( \theta \int_0^1 \int_0^1 \frac{dsdr}{\abs{X_s^1-X_r^1}^{\sigma}} \Big) \Big]
    \leq A^{\frac{1}{p_1}}_1 A_2^{\frac{1}{2p_2}} A_3^{\frac{1}{2p_2}},
  \end{equation*}
  where
  \begin{align*}
    A_1 &\defby \expect\Big[ \exp\Big( p_1 \theta \int_0^{1-\delta} \int_0^{1-\delta} \frac{dsdr}{\abs{X_s^1-X_r^1}^{\sigma}} \Big) \Big],\\
    A_2 &\defby \expect\Big[ \exp\Big( 4 p_2 \theta \int_{1-\delta}^{1} \int_{\delta}^{1} \frac{dsdr}{\abs{X_s^1-X_r^1}^{\sigma}} \Big) \Big],\\
    A_3 &\defby \expect\Big[ \exp\Big( 4 p_2 \theta \int_{1-\delta}^1 \int_0^{\delta} \frac{dsdr}{\abs{X_s^1-X_r^1}^{\sigma}} \Big) \Big].
  \end{align*}
  We evaluate each of $A_1$, $A_2$ and $A_3$.

  We first evaluate $A_1$. By Lemma \ref{lem:girsanov_brownian_bridge},
  \begin{align*}
    A_1 &= \Big(\frac{1}{\delta}\Big)^{\frac{d}{2}}
    \expect\Big[ \exp\Big( p_1 \theta \int_0^{1-\delta} \int_0^{1-\delta} \frac{dsdr}{\abs{B_s-B_r}^{\sigma}} \Big)
    \exp\Big( - \frac{\abs{B_{1-\delta}}^2}{2\delta} \Big) \Big] \\
    &\leq \Big(\frac{1}{\delta}\Big)^{\frac{d}{2}}
    \expect\Big[ \exp\Big( p_1 \theta \int_0^{1} \int_0^{1} \frac{dsdr}{\abs{B_s-B_r}^{\sigma}} \Big) \Big].
  \end{align*}
  Theorem \ref{thm:brownian_intersection} implies
  \begin{equation*}
    \limsup_{\theta \to \infty} \theta^{-\frac{2}{2-\sigma}} \log A_1^{\frac{1}{p_1}} \leq p_1^{\frac{\sigma}{2-\sigma}} \tilde{\kappa}.
  \end{equation*}

  We next evaluate $A_2$. By the reversibility of $(X_t)$, Lemma \ref{lem:girsanov_brownian_bridge} and H\"older's inequality,
  \begin{align*}
    A_2 &= \expect\Big[ \exp\Big( 4 p_2 \theta \int_{0}^{\delta} \int_{0}^{1-\delta} \frac{dsdr}{\abs{X_s^1-X_r^1}^{\sigma}} \Big) \Big] \\
    &\leq \delta^{-\frac{d}{2}}
    \expect\Big[ \exp\Big( 4 p_2 \theta \int_{0}^{\delta} \int_{0}^{1-\delta} \frac{dsdr}{\abs{B_s-B_r}^{\sigma}} \Big) \Big] \\
    &\leq \delta^{-\frac{d}{2}} \prod_{j=1}^{n-1}
    \expect\Big[ \exp\Big( 4 p_2 n \theta \int_{0}^{\delta} \int_{(j-1)\delta}^{j\delta} \frac{dsdr}{\abs{B_s-B_r}^{\sigma}} \Big) \Big]^{\frac{1}{n-1}}.
  \end{align*}
  We claim for $j=2, 3, \ldots, n-1$
  \begin{equation}\label{eq:0_is_greater_than_j}
    \expect\Big[ \exp\Big( 4 p_2 n \theta \int_{0}^{\delta} \int_{(j-1)\delta}^{j\delta} \frac{dsdr}{\abs{B_s-B_r}^{\sigma}} \Big) \Big]
    \leq \expect\Big[ \exp\Big( 2^{2+\frac{\sigma}{2}}
    p_2 n \theta \int_{0}^{\delta} \int_{0}^{\delta} \frac{dsdr}{\abs{B_s-B_r}^{\sigma}} \Big) \Big]
  \end{equation}
  Indeed, if $(\tilde{B}_t)$ is an independent Brownian motion and $g_a$ is the density of the distribution of $B_a$,
  \begin{multline*}
    \expect\Big[ \exp\Big( 4 p_2 n \theta \int_{0}^{\delta} \int_{(j-1)\delta}^{j\delta} \frac{dsdr}{\abs{B_s-B_r}^{\sigma}} \Big) \Big] \\
    = \int_{\R^d} g_{(j-2)\delta}(z)
    \expect\Big[ \exp\Big( 4 p_2 n \theta \int_{0}^{\delta} \int_{0}^{\delta} \frac{dsdr}{\abs{B_s-\tilde{B}_r - z}^{\sigma}} \Big) \Big] dz
  \end{multline*}
  To prove \eqref{eq:0_is_greater_than_j}, it suffices to show
  \begin{equation*}
    \expect\Big[ \Big( \int_0^{\delta} \int_0^{\delta} \frac{dsdr}{\abs{B_s - \tilde{B}_r - z}^{\sigma}} \Big)^k\Big]
    \leq \expect\Big[ \Big( 2^{\frac{\sigma}{2}} \int_0^{\delta} \int_0^{\delta} \frac{dsdr}{\abs{B_s - B_r}^{\sigma}} \Big)^k \Big]
  \end{equation*}
  for every $k \in \N$ and $z \in \R^d$.
  As in the proof of Lemma \ref{lem:lifshitz_tail_lower_bound}, this comes down to checking the inequality
  \begin{align*}
    \MoveEqLeft[3]
    \expect\Big[ \Big(\sum_{j=1}^k \inp{B_{s_j} - \tilde{B}_{r_j} - z}{\xi_j} \Big)^2 \Big] \\
    &= \expect\Big[ \Big( \sum_{j=1}^k \inp{B_{s_j}}{\xi_j} \Big)^2 \Big]
    + \expect\Big[ \Big( \sum_{j=1}^k \inp{\tilde{B}_{r_j}}{\xi_j} \Big)^2 \Big]
    + \Big( \sum_{j=1}^k \inp{z}{\xi_j} \Big)^2 \\
    &\geq \frac{1}{2} \expect\Big[ \Big(\sum_{j=1}^k \inp{B_{s_j} - B_{r_j}}{\xi_j} \Big)^2 \Big].
  \end{align*}
  Thus we end the proof of \eqref{eq:0_is_greater_than_j}.
  As a result,
  \begin{equation*}
    A_2 \leq \delta^{-\frac{d}{2}} \expect\Big[ \exp\Big( 2^{2+\frac{\sigma}{2}} p_2 \delta^{1-\frac{\sigma}{2}} \theta
    \int_0^1 \int_0^1 \frac{dsdr}{\abs{B_s - B_r}^{\sigma}} \Big) \Big],
  \end{equation*}
  and hence Theorem \ref{thm:brownian_intersection} yields
  \begin{equation*}
    \limsup_{\theta \to \infty} \theta^{-\frac{2}{2-\sigma}} \log A_2^{\frac{1}{2p_2}}
    \leq \frac{(2^{2+\frac{\sigma}{2}} p_2 \delta^{1-\frac{\sigma}{2}})^{\frac{2}{2-\sigma}}}{2p_2} \tilde{\kappa}.
  \end{equation*}

  Finally, we evaluate $A_3$.
  The key is to observe that, under the condition that $X_{\frac{1}{2}}^1 = x$,
  $(X_t^1)_{t \in [0, \frac{1}{2}]}$ and $(X_{1-t}^1)_{t \in [0, \frac{1}{2}]}$ are independent and identically distributed and
  the distribution is the Brownian bridge from $0$ at $t=0$ to $x$ at $t=\frac{1}{2}$.
  Therefore, combined with Lemma \ref{lem:girsanov_brownian_bridge}, we obtain
  \begin{equation*}
    A_3 \leq (1-2\delta)^{-d} \int_{\R^d} g_{\frac{1}{4}}(x)
    \expect\Big[ \exp\Big(4p_2 \theta \int_0^{\delta} \int_0^{\delta} \frac{dsdr}{\abs{B_s - \tilde{B}_r}^{\sigma}} \Big)
    \exp\Big( \frac{2\inp{x}{B_{\delta} + \tilde{B}_{\delta}}}{1-2\delta}\Big) \Big]dx.
  \end{equation*}
  For sufficiently small $\delta \in (0, 1)$,
  \begin{equation*}
    \int_{\R^d} g_{\frac{1}{4}}(x) \expect\Big[ \exp\Big( \frac{4\inp{x}{B_{\delta} + \tilde{B}_{\delta}}}{1-2\delta}\Big) \Big] < \infty.
  \end{equation*}
  On the other hand, as in the evaluation of $A_2$, we have
  \begin{equation*}
    \expect\Big[ \exp\Big(8 p_2 \theta \int_0^{\delta} \int_0^{\delta} \frac{dsdr}{\abs{B_s - \tilde{B}_r}^{\sigma}} \Big)\Big]
    \leq \expect\Big[ \exp\Big(2^{3+\frac{\sigma}{2}} p_2 \delta^{2-\frac{\sigma}{2}} \theta
    \int_0^{1} \int_0^{1} \frac{dsdr}{\abs{B_s - B_r}^{\sigma}} \Big) \Big].
  \end{equation*}
  Therefore, Theorem \ref{thm:brownian_intersection} yields
  \begin{align*}
    \limsup_{\theta \to \infty} \theta^{-\frac{2}{2-\sigma}} \log A_3^{\frac{1}{2p_2}}
    &\leq \limsup_{\theta \to \infty} \theta^{-\frac{2}{2-\sigma}} \log
    \expect\Big[ \exp\Big(2^{3+\frac{\sigma}{2}} p_2 \delta^{2-\frac{\sigma}{2}} \theta
      \int_0^{1} \int_0^{1} \frac{dsdr}{\abs{B_s - B_r}^{\sigma}} \Big) \Big]^{\frac{1}{4p_2}} \\
      &= \frac{(2^{3+\frac{\sigma}{2}} p_2 \delta^{2-\frac{\sigma}{2}})^{\frac{2}{2-\sigma}}}{4p_2} \tilde{\kappa}.
  \end{align*}

  Consequently, we obtain
  \begin{multline*}
    \limsup_{\theta \to \infty} \theta^{-\frac{2}{2-\sigma}}
    \expect\Big[ \exp \Big( \theta \int_0^1 \int_0^1 \frac{dsdr}{\abs{X_s^1-X_r^1}^{\sigma}} \Big) \Big] \\
    \leq \Big\{ p_1^{\frac{\sigma}{2-\sigma}}
    + \frac{(2^{2+\frac{\sigma}{2}} p_2 \delta^{1-\frac{\sigma}{2}})^{\frac{2}{2-\sigma}}}{2p_2}
    + \frac{(2^{3+\frac{\sigma}{2}} p_2 \delta^{2-\frac{\sigma}{2}})^{\frac{2}{2-\sigma}}}{4p_2} \Big\} \tilde{\kappa}.
  \end{multline*}
  Letting $\delta \to 0$ (equivalently $n \to \infty$) and then $p_1 \to 1$, we complete the proof.
\end{proof}
\begin{proposition}[{\cite[Corollary 2]{nagai_remark_1977}, \cite[Theorem 4.12.7]{bingham_goldie_teugels_1987}}]
  \label{prop:minlos_tauberian}
  Let $\rho$ be a distribution function on $\R$ and $k$ be its Laplace transform
  \begin{equation*}
    k(t) \defby \int_{-\infty}^{\infty} e^{-t\lambda} \rho(d\lambda).
  \end{equation*}
  Then, the following two conditions are equivalent:
  \begin{enumerate}[(i)]
    \item $\lim_{\lambda \to -\infty} \abs{\lambda}^{-\alpha} \log \rho(\lambda) = -A, \,\, \alpha > 1, A>0$,
    \item $\lim_{t \to \infty} t^{-\gamma} \log k(t) = B, \,\, \gamma > 1, B > 0$,
  \end{enumerate}
  where $\alpha, \gamma, A, B$ are related by
  \begin{equation*}
    \alpha = \frac{\gamma}{\gamma - 1}, \quad A = (\gamma - 1) \gamma^{\frac{\gamma}{1-\gamma}} B^{\frac{1}{1-\gamma}}.
  \end{equation*}
\end{proposition}
\begin{proof}[Proof of Theorem \ref{thm:lifshitz_tail_riesz}]
  By Lemma \ref{lem:lifshitz_tail_lower_bound} and Lemma \ref{lem:lifshitz_tail_upper_bound}, we have
  \begin{equation*}
    \lim_{\theta \to \infty} \theta^{-\frac{2}{2-\sigma}} \log
    \expect\Big[ \exp\Big( \theta \int_0^1 \int_0^1 \frac{ds dr}{\abs{X_s^1 - X_r^1}^{\sigma}} \Big) \Big] =
    2^{\frac{6}{2-\sigma}} (2-\sigma) (4-\sigma)^{-\frac{4-\sigma}{2-\sigma}} \rho(d, \sigma)^{\frac{2}{2-\sigma}}.
  \end{equation*}
  Since $\int_0^t \int_0^t \abs{X^t_s - X^t_r}^{-\sigma} ds dr \dequal t^{2-\frac{\sigma}{2}}
  \int_0^1 \int_0^1 \abs{X_s^1 - X_r^1}^{-\sigma} ds dr$ by
  the scaling property, we obtain
  \begin{multline*}
    \lim_{t \to \infty} t^{-\frac{4-\sigma}{2-\sigma}} \log \expect
    \Big[ \exp \Big( \frac{\nu}{2} \int_0^t \int_0^t \frac{ds dr}{\abs{X^t_s - X^t_r}^{\sigma}} \Big) \Big] \\
    = \Big(\frac{\nu}{2}\Big)^{\frac{2}{2-\sigma}}
    2^{\frac{6}{2-\sigma}} (2-\sigma) (4-\sigma)^{-\frac{4-\sigma}{2-\sigma}} \rho(d, \sigma)^{\frac{2}{2-\sigma}}.
  \end{multline*}
  Now the claim follows from the identity \eqref{eq:laplace_transform_of_N_V} and Proposition \ref{prop:minlos_tauberian}.
\end{proof}
\section{Gaussian computations}\label{sec:gaussian}
We use the notation introduced in Section \ref{sec:pam_and_ah}.
\begin{proposition}\label{prop:ergodicitiy_of_white_noise}
  The family of maps $\{T_x\}_{x \in \R^2}$ is measure preserving and ergodic. That is, for any measurable set $A$,
  $\wnP(T_x^{-1}A) = \wnP(A)$ and the condition $T_x^{-1} A = A$ for every $x \in \R^2$ implies $\wnP(A) \in \{0,1\}$.
\end{proposition}
\begin{proof}
  The map $T_x$ is obviously measure preserving.
  We observe for $\phi_1, \ldots, \phi_n \in \S$
  \begin{multline*}
    \wnexpect[e^{i \sum_{j=1}^n a_j \inp{\xi}{\phi_j}} e^{i \sum_{j=1}^n b_j \inp{T_x\xi}{\phi_j}}]\\
    = \wnexpect[e^{i \sum_{j=1}^n a_j \inp{\xi}{\phi_j}}] \wnexpect[e^{i \sum_{j=1}^n b_j \inp{\xi}{\phi_j}}]
    \exp\Big( \sum_{j,k=1}^n a_j b_k \inp{\phi_j}{\phi_k(\cdot - x)} \Big).
  \end{multline*}
  Therefore,
  \begin{equation}\label{eq:mixing_fourier}
    \begin{multlined}
      \lim_{x \to \infty} \wnexpect[ f(\inp{\xi}{\phi_1}, \ldots, \inp{\xi}{\phi_n}) g(\inp{T_x\xi}{\phi_1}, \ldots, \inp{T_x\xi}{\phi_n})]\\
      = \wnexpect[ f(\inp{\xi}{\phi_1}, \ldots, \inp{\xi}{\phi_n})] \wnexpect[g(\inp{\xi}{\phi_1}, \ldots, \inp{\xi}{\phi_n})]
    \end{multlined}
  \end{equation}
  if $f(x_1, \ldots, x_n) = e^{i \sum_{j=1}^n a_j x_j}$ and $g(x_1, \ldots, x_n) = e^{i \sum_{j=1}^n b_j x_j}$.
  This shows the convergence in law of $(\inp{\xi}{\phi_i}, \inp{T_x \xi}{\phi_i})_{i=1}^n$ to
  $(\inp{\xi}{\phi_i}, \inp{\tilde{\xi}}{\phi_i})_{i=1}^n$, where $\tilde{\xi}$ is an independent copy of $\xi$.
  Therefore, \eqref{eq:mixing_fourier} holds for polynomials $f$ and $g$.

  Let $\{e_n\}_{n=1}^{\infty} \subseteq \S$ be the orthonormal basis of $L^2(\R^2)$.
  Since the Cameron-Martin space of the white noise is $L^2(\R^2)$,
  the Wiener chaos decomposition \cite[Theorem 1.1.1]{nualart2006malliavin}
  implies that for given $F, G \in L^2(\wnP)$ there exist polynomials $f_n, g_n: \R^{n} \to \R$ such that
  \begin{align*}
    &\lim_{n \to \infty} \wnexpect[ \abs{F - f_n(\inp{\xi}{e_1}, \ldots, \inp{\xi}{e_{n}})}^2] = 0,\\
    &\lim_{n \to \infty} \wnexpect[ \abs{G - g_n(\inp{\xi}{e_1}, \ldots, \inp{\xi}{e_{n}})}^2] = 0.
  \end{align*}
  Since \eqref{eq:mixing_fourier} holds for $f=f_n$, $g=g_n$ with $(\phi_1, \ldots, \phi_n)$ replaced by
  $(e_1, \ldots, e_{n})$, we obtain
  \begin{multline*}
    \limsup_{x \to \infty} \abs*{ \wnexpect[F(\xi) G(T_x \xi)] - \wnexpect[F] \wnexpect[G]}\\
    \lesssim \norm{F - f_n(\inp{\xi}{e_1}, \ldots, \inp{\xi}{e_{n}})}_{L^2(\wnP)}
    + \norm{G - g_n(\inp{\xi}{e_1}, \ldots, \inp{\xi}{e_{n}})}_{L^2(\wnP)}
  \end{multline*}
  for arbitrary $n$, and hence $\lim_{x \to \infty} \wnexpect[F(\xi) G(T_x \xi)] = \wnexpect[F] \wnexpect[G]$.

  If $A \subseteq \S'$ is invariant under the action of $\{T_x\}_{x \in \R^2}$, then $\wnP(A) = \wnP(A)^2$ and hence
  $\wnP(A) \in \{0, 1\}$.
\end{proof}

Let $\{\rho_i\}_{i=-1}^{\infty}$ be a dyadic partition of unity on $\R^2$, i.e., $\rho_{-1}$ and $\rho_0$ are smooth nonnegative
radial function on
$\R^2$, $\rho_{-1}$ is supported on a ball, $\rho_0$ is supported on an annulus, $\rho_j = \rho_0(2^{-j} \cdot)$ for $j \in \N$
and the following conditions are satisfied:
\begin{equation*}
  \sum_{j=-1}^{\infty} \rho_j(y) = 1, \quad \mbox{for every }y \in \R^2,
\end{equation*}
\begin{equation*}
  \abs{j-k} \geq 2 \implies \operatorname{supp}(\rho_j) \cap \operatorname{supp}(\rho_k) = \emptyset.
\end{equation*}
We set $\Delta_i w \defby \sum_{k \in \N} \rho_i(k/L) \inp{w}{\n_{k,L}} \n_{k, L}$ and
$\rho^{\resonant}(k, l) \defby \sum_{\abs{i-j} \leq 1} \rho_i(k) \rho_j(l)$.
We have for every $\gamma \in (0, \infty)$
\begin{equation}\label{eq:estimate_of_rho}
  \rho_i(x) \lesssim_{\gamma} \Big( \frac{2^i}{1 + \abs{x}} \Big)^{\gamma}, \qquad i \geq -1, \,\, x \in \R^2.
\end{equation}
As shown in the proof of \cite[Lemma 11.11]{chouk2020asymptotics}, we have for $\gamma \in (0, 1)$,
$L \in [1, \infty)$ and $y \in Q_L$,
\begin{equation}\label{eq:estimate_of_rho_resonant}
   \sum_{k, l \in \N_0^2} \frac{\rho^{\resonant}(\frac{k}{L}, \frac{l}{L})^2 (\abs{k}^{\gamma} + \abs{l}^{\gamma})}
  {(1 + \pi^2 L^{-2} \abs{l}^2)^2} \abs{\Delta_i(\n_{k,L} \n_{l, L})(y)}^2  \lesssim_{L, \gamma} 2^{\gamma i}.
\end{equation}
We set
\begin{equation*}
  T_x h_{L, \epsilon}(k, l) \defby \wnexpect[\inp{T_x \xi_{\epsilon}}{\n_{k,L}} \inp{\xi_{\epsilon}}{\n_{l,L}}],
  \qquad h_{L, \epsilon}(k, l) \defby T_0 h_{L, \epsilon}(k, l).
\end{equation*}
\begin{lemma}\label{lem:estimate_of_h}
  There exists $C \in (0, \infty)$ such that
  for every $x \in \R^2$, $L \in [1, \infty)$, $\epsilon \in (0, 1)$ and $k, l \in \N_0^2$,
  \begin{equation*}
    \abs{T_x h_{L, \epsilon}(k,l)} \leq 1, \quad \abs{T_x h_{L, \epsilon}(k, l) - h_{L, \epsilon}(k, l)} \leq C \abs{k} \abs{x}.
  \end{equation*}
\end{lemma}
\begin{proof}
  Since
  \begin{equation*}
    \abs{T_x h_{L, \epsilon}(k,l)} \leq \expect[\inp{\xi_{\epsilon}}{\n_{k,L}}^2]^{\frac{1}{2}}
    \expect[\inp{\xi_{\epsilon}}{\n_{l,L}}^2]^{\frac{1}{2}}
  \end{equation*}
  and $\expect[\inp{\xi_{\epsilon}}{\n_{k,L}}^2] = \norm{p_{\frac{\epsilon}{2}} \conv \n_{k,L}}_{L^2(\R^2)} \leq
  \norm{\n_{k,L}}_{L^2(Q_L)} = 1$, we have the first claimed inequality. For the second inequality, as we have
  \begin{equation*}
    T_x h_{\epsilon, L}(k, l) =
    \int_{\R^2 \times \R^2} p_{\epsilon}(y) \n_{k, L}(z - x) \n_{l, L}(z + y) \indic_{Q_L}(z-x) \indic_{Q_L}(z+y) dy dz,
  \end{equation*}
  it suffices to note
  \begin{equation*}
    \Big\lvert \int_{\R^2} \n_{k,L}(z-x) \n_{l, L}(z+y) (\indic_{Q_L}(z-x) - \indic_{Q_L}(z)) \indic_{Q_L}(z+y) dz \Big\rvert
    \leq C \abs{x},
  \end{equation*}
  \begin{equation*}
    \Big\lvert \int_{\R^2} (\n_{k,L}(z-x) - \n_{k,L}(z)) \n_{l, L}(z+y) \indic_{Q_L}(z) \indic_{Q_L}(z+y) dz \Big\rvert
    \leq C \abs{k} \abs{x}. \qedhere
  \end{equation*}
\end{proof}
\begin{proposition}\label{prop:unif_conv_of_enhanced_noise}
  For each $\alpha \in (-\frac{4}{3}, -1)$ and $L, M \in [1, \infty)$, we have
  \begin{equation*}
      \lim_{\epsilon, \epsilon' \to 0+} \sup_{x \in Q_M}
      \norm{T_x \bm{\xi}_{L, \epsilon}^{\n} - T_x \bm{\xi}_{L, \epsilon'}^{\n}}_{\X_{\n, L}^{\alpha}} = 0
      \quad \mbox{in } \wnP\mbox{-probability.}
  \end{equation*}
\end{proposition}
\begin{proof}
  Since
  \begin{equation*}
    \Delta_i(T_x \xi_{L, \epsilon} - \xi_{L, \epsilon})(y)
    = \sum_{k \in \N_0^2} \rho_i(L^{-1}k) (\inp{T_x \xi_{\epsilon}}{\n_{k,L}} - \inp{\xi_{\epsilon}}{\n_{k,L}}) \n_{k, L},
  \end{equation*}
  we have
  \begin{multline*}
    \wnexpect[ \abs{\Delta_i(T_x \xi_{L, \epsilon} - \xi_{L, \epsilon})(y)}^2 ]
    = \sum_{k, l \in \N_0^2} \rho_i(L^{-1}k) \rho_i(L^{-1} l) \n_{k, L}(y) \n_{l, L}(y) \\
    \times \wnexpect\big[ (\inp{T_x \xi_{\epsilon}}{\n_{k,L}} - \inp{\xi_{\epsilon}}{\n_{k,L}})
    (\inp{T_x \xi_{\epsilon}}{\n_{l,L}} - \inp{\xi_{\epsilon}}{\n_{l,L}})].
  \end{multline*}
  By Lemma \ref{lem:estimate_of_h}, for $\delta \in (0, 1)$,
  \begin{align*}
    \MoveEqLeft[3]
    \wnexpect\big[ (\inp{T_x \xi_{\epsilon}}{\n_{k,L}} - \inp{\xi_{\epsilon}}{\n_{k,L}})
    (\inp{T_x \xi_{\epsilon}}{\n_{l,L}} - \inp{\xi_{\epsilon}}{\n_{l,L}})] \\
    &= \abs{h_{L, \epsilon}(k, l) - T_x h_{L, \epsilon}(k, l) - T_xh_{L, \epsilon}(l, k) + h_{L, \epsilon}(l, k)} \\
    &\lesssim (\abs{k}^{\delta} + \abs{l}^{\delta}) \abs{x}^{\delta}.
  \end{align*}
  Therefore, by \eqref{eq:estimate_of_rho},
  \begin{align*}
    \wnexpect[ \abs{\Delta_i(T_x \xi_{L, \epsilon} - \xi_{L, \epsilon})(y)}^2 ]
    &\lesssim_{L, \delta} 2^{4 i (1 + \delta)} \abs{x}^{\delta} \sum_{k, l \in \N_0^2}
    \frac{\abs{k}^{\delta} + \abs{l}^{\delta}}{(1 + \abs{k})^{2 + 2\delta} ( 1 + \abs{l} )^{2 + 2\delta}}\\
    &\lesssim_{\delta} 2^{4i(1+\delta)} \abs{x}^{\delta}.
  \end{align*}
  On the other hand, we have
  \begin{equation*}
    \wnexpect[ \abs{\Delta_i (\xi_{L, \epsilon} - \xi_{L, \epsilon'})(y)}^2 ]
    \leq 2^{2i(1+\delta)} b(\epsilon, \epsilon'),
  \end{equation*}
  where $b(\epsilon, \epsilon') \defby \expect[ \norm{\xi_{L, \epsilon} - \xi_{L, \epsilon'}}_{\csp^{-1-\delta}_{\n, L}}^2 ]$.
  According to \cite[Theorem 6.1, Theorem 6.2 and Lemma 6.14]{konig2020longtime}, we have
  $\lim_{\epsilon, \epsilon' \to 0+} b(\epsilon, \epsilon') = 0$.
  Therefore, combining these two estimates,
  \begin{align*}
    \MoveEqLeft[5]
    \wnexpect\big[ \abs{\Delta_i\{(T_x \xi_{L, \epsilon}^{\n} - T_y \xi_{L, \epsilon}^{\n}) -
    (T_x \xi_{L, \epsilon'}^{\n} - T_y \xi_{L, \epsilon'}^{\n}) \}}^2 \big] \\
    &= \wnexpect\big[ \abs{\Delta_i\{(T_x \xi_{L, \epsilon}^{\n} - T_y \xi_{L, \epsilon}^{\n}) -
    (T_x \xi_{L, \epsilon'}^{\n} - T_y \xi_{L, \epsilon'}^{\n})\}}^2 \big]^{\delta} \\
    &\hspace{1cm} \times \wnexpect\big[ \abs{\Delta_i\{(T_x \xi_{L, \epsilon}^{\n} - T_y \xi_{L, \epsilon}^{\n}) -
    (T_x \xi_{L, \epsilon'}^{\n} - T_y \xi_{L, \epsilon'}^{\n})\}}^2 \big]^{1-\delta}\\
    &\lesssim_{L, \delta} 2^{2i(1+\delta)^2} b(\epsilon, \epsilon')^{1- \delta} \abs{x-y}^{\delta^2}.
  \end{align*}
  Then, by Lemma \cite[Lemma 6.10]{chouk2020asymptotics},
  \begin{equation*}
    \wnexpect[ \norm{(T_x \xi_{L, \epsilon}^{\n} - T_y \xi_{L, \epsilon}^{\n}) -
    (T_x \xi_{L, \epsilon'}^{\n} - T_y \xi_{L, \epsilon'}^{\n})}_{\csp_{\n, L}^{-(1+\delta)^2 - \delta - \frac{2}{p}}}^p]
    \lesssim_{L, \delta, p} b(\epsilon, \epsilon')^{\frac{p(1-\delta)}{2}} \abs{x-y}^{\frac{p\delta^2}{2}}
  \end{equation*}
  for every $p \geq 2$. If $p \delta^2 > 4$, the multidimensional Kolmogorov continuity theorem implies
  \begin{equation*}
    \wnexpect\big[ \sup_{x \in Q_M} \norm{T_x \xi_{L, \epsilon}^{\n}  -
    T_x \xi_{L, \epsilon'}^{\n} }_{\csp_{\n, L}^{-(1+\delta)^2 - \delta - \frac{2}{p}}}]
    \lesssim_{L, \delta, p, M} b(\epsilon, \epsilon')^{2(1-\delta)}.
  \end{equation*}
  Supposing $\delta \in (0, 1)$ are sufficiently small and $p$ is sufficiently large, we have
  the convergence of the first component.

  The proof of the convergence of the second component is similar but more complicated.
  We have
  \begin{multline*}
    T_x \Xi_{L, \epsilon}^{\n} - \Xi_{L, \epsilon}^{\n}
    = \sum_{k, l \in \N_0^2} \rho^{\resonant}(L^{-1}k, L^{-1}l) (1+\pi^2 L^{-2} \abs{l}^2 )^{-1} \\
    \times (\inp{T_x \xi_{\epsilon}}{\n_{k,L}} \inp{T_x \xi_{\epsilon}}{\n_{l, L}} -
    \inp{ \xi_{\epsilon}}{\n_{k,L}} \inp{ \xi_{\epsilon}}{\n_{l, L}}) \n_{k, L} \n_{l, L}.
  \end{multline*}
  Therefore,
  \begin{align*}
    \MoveEqLeft[3]
    \wnexpect[ \abs{\Delta_i(T_x \Xi_{L, \epsilon}^{\n} - \Xi_{L, \epsilon}^{\n})(y)}^2 ] \\
    &= \sum_{k_1, k_2, l_1, l_2 \in \N_0^2}
    \Big\{ \prod_{j=1}^2 \rho^{\resonant}(L^{-1} k_j, L^{-1} l_j) (1 + \pi^2 L^{-2} \abs{l_j}^2)^{-1}
    \Delta_i(\n_{k_j, L} \n_{l_j, L}) (y) \Big\} \\
    &\hspace{3cm} \times \wnexpect\Big[ \prod_{j=1}^2 (\inp{T_x \xi_{\epsilon}}{\n_{k_j,L}} \inp{T_x \xi_{\epsilon}}{\n_{l_j, L}} -
    \inp{ \xi_{\epsilon}}{\n_{k_j,L}} \inp{ \xi_{\epsilon}}{\n_{l_j, L}})\Big].
  \end{align*}
  By Wick's Theorem,
  \begin{align*}
    \MoveEqLeft[5]
    \wnexpect[ \inp{T_x \xi_{\epsilon}}{\n_{k, L}}  \inp{T_x \xi_{\epsilon}}{\n_{l, L}}
    \inp{\xi_{\epsilon}}{\n_{k, L}} \inp{\xi_{\epsilon}}{\n_{l, L}}] \\
    =& \wnexpect[\inp{T_x \xi_{\epsilon}}{\n_{k, L}} \inp{T_x \xi_{\epsilon}}{\n_{l, L}}]
    \wnexpect[\inp{\xi_{\epsilon}}{\n_{k, L}} \inp{ \xi_{\epsilon}}{\n_{l, L}}] \\
    &+\wnexpect[\inp{T_x \xi_{\epsilon}}{\n_{k, L}} \inp{\xi_{\epsilon}}{\n_{k, L}}]
    \wnexpect[\inp{T_x \xi_{\epsilon}}{\n_{l, L}} \inp{ \xi_{\epsilon}}{\n_{l, L}}] \\
    &+\wnexpect[\inp{T_x \xi_{\epsilon}}{\n_{k, L}} \inp{\xi_{\epsilon}}{\n_{l, L}}]
    \wnexpect[\inp{T_x \xi_{\epsilon}}{\n_{l, L}} \inp{ \xi_{\epsilon}}{\n_{k, L}}] \\
    =& h_{L, \epsilon}(k, l)^2 + T_xh_{L, \epsilon}(k,k) T_xh_{L, \epsilon}(l, l)
    + T_x h_{L, \epsilon}(k, l) T_x h_{L, \epsilon}(l, k).
  \end{align*}
  By Lemma \ref{lem:estimate_of_h}, we obtain for $\delta \in (0, 1)$
  \begin{align*}
    \MoveEqLeft[3]
    \wnexpect\big[ (\inp{T_x \xi_{\epsilon}}{\n_{k, L}}\inp{T_x \xi_{\epsilon}}{\n_{l, L}}
    - \inp{\xi_{\epsilon}}{\n_{k, L}} \inp{\xi_{\epsilon}}{\n_{k, L}})^2 \big] \\
    =&2 \big\{ h_{L, \epsilon}(k, l)^2 + h_{L, \epsilon}(k,k) h_{L, \epsilon}(l, l)
    +  h_{L, \epsilon}(k, l)  h_{L, \epsilon}(l, k) \big\} \\
    &- 2 \big\{ h_{L, \epsilon}(k, l)^2 + T_xh_{L, \epsilon}(k,k) T_xh_{L, \epsilon}(l, l)
    + T_x h_{L, \epsilon}(k, l) T_x h_{L, \epsilon}(l, k) \big\} \\
    \lesssim& (\abs{k}^{\delta} + \abs{l}^{\delta}) \abs{x}^{\delta}.
  \end{align*}
  Thus, by \eqref{eq:estimate_of_rho_resonant},
  \begin{align*}
    \MoveEqLeft[3]
    \wnexpect[ \abs{\Delta_i(T_x \Xi_{L, \epsilon}^{\n} - \Xi_{L, \epsilon}^{\n})(y)}^2 ] \\
    &\leq \sum_{k_1, k_2, l_1, l_2 \in \N_0^2}
    \Big\{ \prod_{j=1}^2 \rho^{\resonant}(L^{-1} k_j, L^{-1} l_j) (1 + \pi^2 L^{-2} \abs{l_j}^2)^{-1}
    \abs{\Delta_i(\n_{k_j, L} \n_{l_j, L}) (y)}  \\
    &\hspace{3cm} \times \wnexpect\Big[  (\inp{T_x \xi_{\epsilon}}{\n_{k_j,L}} \inp{T_x \xi_{\epsilon}}{\n_{l_j, L}} -
    \inp{ \xi_{\epsilon}}{\n_{k_j,L}} \inp{ \xi_{\epsilon}}{\n_{l_j, L}})^2\Big]^{\frac{1}{2}} \Big\}\\
    &\lesssim \abs{x}^{\delta} \Big\{
    \sum_{k, l \in \N_0^2} \frac{\rho^{\resonant}(\frac{k}{L}, \frac{l}{L})^2 (\abs{k}^{\delta} + \abs{l}^{\delta})}
    {(1 + \pi^2 L^{-2} \abs{l}^2)^2} \abs{\Delta_i(\n_{k,L} \n_{l, L})(y)}^2 \Big\}^2 \\
    &\lesssim_{L, \delta} \abs{x}^{\delta} 2^{2 \delta i}
  \end{align*}
  The remaining estimate can be completed as in the first half of the proof.
\end{proof}
\section{Two lemmas on Laplace transforms}\label{sec:laplace_transform}
\begin{lemma}\label{lem:laplace_transform_of_probability_measures}
  If a set of probability measures $\{\nu_n\}_{n=1}^{\infty}$ on $\R$ satisfies
  \begin{equation*}
    \sup_{n \geq 1} \int_{\R} e^{T \abs{\lambda}} \nu_{n}(d\lambda) < \infty
  \end{equation*}
  for some $T \in (0, \infty)$
  and the Laplace transform $g_n(t) \defby \int_{\R} e^{- t \lambda} \nu_n(d\lambda)$ converges in $[0, \infty)$ for $t$
  in a dense subset of $(-T, T)$, then $\{\nu_n\}_{n=1}^{\infty}$ converges weakly to a probability measure $\nu$ and
  \begin{equation*}
    \lim_{n \to \infty} \int_{\R} e^{-t\lambda} \nu_n(d\lambda) = \int_{\R} e^{-t\lambda} \nu(d\lambda),
    \quad t \in (-T, T).
  \end{equation*}
\end{lemma}
\begin{proof}
  We regard $g_n$ as an analytic function in $\set{z \in \C \given \abs{\Re(z)} < T }$. Since $g_n$ is uniformly bounded,
  Montel's theorem implies that $\{g_n\}_{n=1}^{\infty}$ is normal. Since $g_n(t)$ converges in a dense subset of $(-T, T)$, uniqueness of
  analytic continuation implies that the sequence $\{g_n\}_{n=1}^{\infty}$ converges to $g$ uniformly on each compact set.
  Then, Bochner's theorem implies that there exists a probability measure $\nu$ such that
  \begin{equation*}
    \int_{\R} e^{-it \lambda} \nu(d\lambda) = g(it)
  \end{equation*}
  and the claim follows.
\end{proof}
\begin{lemma}\label{lem:analytic_extension_of_Laplace_transform}
  Let $0 < T_1 < T_2 < \infty$, $\nu$ be a finite measure on $[0, \infty)$ with
  $\int_{[0, \infty)} e^{t\lambda} \nu(d\lambda) < \infty$ for $t < T_1$, and $g$ be an analytic function
  in $\set{z \in \C \given \Re(z) \in (0, T_2)}$.
  If the identity
  \begin{equation}\label{eq:Laplace_transform_equals_g}
    \int_{[0, \infty)} e^{t\lambda} \nu(d\lambda) = g(t)
  \end{equation}
  holds for $t \in (0, T_1)$, then it holds for $t \in (0, T_2)$.
\end{lemma}
\begin{proof}
  Set $T \defby \sup \set{t \in (0, T_2) \given \int_{[0, \infty)} e^{t\lambda} \nu(d\lambda) < \infty}$.
  By uniqueness of analytic extension, it suffices to show $T = T_2$. Assume $T < T_2$.
  Take $\tau \in (0, T)$ and $r \in (0, \tau)$ with $\tau + r < T_2$.
  As \eqref{eq:Laplace_transform_equals_g} holds for $t \in (0, T)$, we have
  \begin{equation*}
    \int_{[0, \infty)} \lambda^k e^{\tau \lambda} \nu(d\lambda) = g^{(k)}(\tau),
  \end{equation*}
  where $g^{(k)}$ is the $k$ th derivative of $g$. Then, by Fubini-Tonelli theorem,
  \begin{equation*}
    \int_{[0, \infty)} e^{(\tau + r) \lambda} \nu(d\lambda)
    = \sum_{k=0}^{\infty} \frac{r^k}{k!} \int_{[0, \infty)} \lambda^k e^{\tau \lambda} \nu(d\lambda)
    = \sum_{k=0}^{\infty} \frac{r^k}{k!} g^{(k)}(\tau) = g(\tau + r) < \infty.
  \end{equation*}
  Since $T < \tau + r$, this contradicts the definition of $T$. Thus, $T = T_2$.
\end{proof}
\section*{Acknowledgment}
The author is very grateful to Nicolas Perkowski and Willem van Zuijlen for having valuable discussions and providing helpful comments.
The author also gratefully thanks to the anonymous referees for numerous valuable comments,
which have significantly improved the presentation of the paper.
The author gratefully acknowledges financial support by the DFG via the IRTG 2544.
\printbibliography
\end{document}